\documentclass[11pt, a4paper]{amsart}
\usepackage[english]{babel}
\usepackage[]{graphicx}
\usepackage{amsfonts,amsmath}	
\usepackage{epsfig}
\usepackage{amsthm,amssymb}
\usepackage[subnum]{cases}
\usepackage{url}
\usepackage{verbatim}
\usepackage{texdraw}
\makeindex

\theoremstyle{definition}
\newtheorem{thm}{Theorem}
\newtheorem{cor}[thm]{Corollary}
\newtheorem{lem}[thm]{Lemma}
\newtheorem{defi}[thm]{Definition}
\newtheorem{prop}[thm]{Proposition}
\newtheorem{exam}[thm]{Example}
\newtheorem{rem}[thm]{Remark}
\newtheorem{ques}[thm]{Question}

\providecommand{\WileyBibTextsc}{}
\let\textsc\WileyBibTextsc

\title[\resizebox{4.5in}{!}{Gonality of complete graphs with a small number of omitted edges}]{Gonality of complete graphs with a small number of omitted edges}
\author{Marta Panizzut}
\email{{\tt marta.panizzut@wis.kuleuven.be}}
\address{KU Leuven, Department of Mathematics, Celestijnenlaan 200B, 3001 Heverlee, Belgium}
\thanks{The author is supported by the Research Project G.0939.13N of the Research Foundation - Flanders (FWO)}

\begin{document}

\begin{abstract}
Let $K_d$ be the complete metric graph on $d$ vertices. We compute the gonality of graphs obtained from $K_d$ by omitting edges forming a $K_h$, or general configurations of at most $d-2$ edges. We also investigate if these graphs can be lifted to curves with the same gonality. We lift the former graphs and the ones obtained by removing up to $d-2$ edges not forming a $K_3$  using models of plane curves with certain singularities. We also study the gonality when removing $d-1$ edges not forming a $K_3$. We use harmonic morphism to lift these graphs to curves with the same gonality because in this case plane singular models can no longer be used due to a result of Coppens and Kato.  \end{abstract}
\maketitle                   

\section{Introduction}

A theory of linear systems of divisors on finite graphs, in analogy with the one on algebraic curves, has been developed by Baker and Norine in \cite{BN}.  In their groundbreaking paper they proved the combinatorial equivalents of fundamental results of the classical theory, such as the Riemann--Roch Theorem  and  the Clifford's inequality.  Since then this setting has been a fruitful ground for further developments. Among others, Gathmann and Kerber \cite{GK}, Mikhalkin and Zarkov \cite{MZ} extended the theory to metric graphs and tropical curves, Amini amd Caporaso \cite{AC} to weighted tropical curves, and  Amini and Baker \cite{AB} to metrized complexes of algebraic curves.  
 
The analogy between the case of curves and graphs is not only purely formal but there is an interplay between them, as explained by Baker in~\cite{Bak}. Let $R$ be a discrete valuation ring with field of fractions $K$. A model for a curve $X$ over $K$ is a surface over $R$ whose generic fiber is $X$. Given a smooth curve $X$ over $K$ and a strongly semistable regular model $\mathfrak{X}$  over $R$ of $X$, it is possible to specialize a divisor on the curve to a divisor on the dual  graph of the special fiber of $\mathfrak{X}$. By this specialization the rank can only increase. When working with an algebraically closed field, complete with respect to a nontrivial non-Archimedean valuation, the same result can be obtained via the analytification $X^{\textrm{an}}$ of the curve $X$, without requiring that the model $\mathfrak{X}$ is regular. The dual graph of the special fiber in this context is called skeleton. 

\vspace{\baselineskip}

Understanding how deep the analogy is between the two theories is an interesting problem. Natural questions that can be considered are whether results for linear systems on curves hold also for metric graphs. Works in this direction showed for example, that the theory of linear systems of Clifford index $0$ is the same for graphs and curves (see \cite{Fac, Cop}), while the same is not true for small Clifford indexes  different from $0$ (see \cite{Cop2}). 

Moreover these problems can be studied together with other natural questions that arise in the setting of Baker's specialization lemma and that are called lifting problems. There are different formulations depending, for example, on whether we are interested in one divisor or all the divisors on a graph. In the former case the question is  whether given a  divisor on a metric graph with a certain rank there exists a divisor on a curve with the same rank specializing to it, see Definition \ref{lifting} and Question \ref{question1}. In the latter one, the problem is, given a metric graph, under which conditions there exists a curve such that for every divisor on the graph there is a divisor on the curve with the same rank specializing to it, see Question \ref{question2}. For hyperelliptic graphs and non-hyperelliptic graphs of genus $3$, Kawaguchi-Yamaki in \cite{KY1, KY2} gave a characterization of the graphs for which it is always possible to lift divisors. Lifting problems can be generalized to the question of understanding to which extent linear systems on the graph $\Gamma$ can be explained by linear systems on the curve $X$. 

\vspace{\baselineskip}
One of the most renowned birational invariants of an algebraic curve is the gonality, that is the smallest integer $d$ for which there exists a $g_d^1$. In this paper we will address these types of problems regarding the gonality for a special type of curves and graphs: plane curves with nodes, and complete graphs with some edges removed. 

The complete graph on $d$ vertices is the dual graph of the special fiber of a model of a smooth plane curve of degree $d$. The special fiber of this model is a reduced union of lines in the plane, see Section \ref{pcs}. The curve and the graph have the same gonality, namely $d-1$. Moreover the divisors of rank $1$ and degree $d-1$ that we will consider on the graph lift to divisor on the curve. Starting from this strong analogy, we want to consider the gonality of singular plane curves.  We will focus in particular on the result proved by Coppens and Kato in \cite{CK} that if $k >0$ and $d \geq 2(k+1)$, the normalization of a plane curves of degree $d$ with $\delta < kd - (k+1)^2 + 3$ nodes has gonality $d-2$. 
 
 Omitting edges of a complete graph can be interpreted as resolving some singularities in the plane curve. In particular the graph $K_d'$ obtained by removing from $K_d$ the edges forming a $K_h$ is the dual graph of the special fiber of a model of a plane curve with a singular point of multiplicity $h$, after blowing up the singular point. The special fiber is a partial normalization of a reduced union of lines in the plane, see Section \ref{mulh}. The curve and the graph have gonality $d-h$ and there are divisors of degree $d-h$ and rank $1$ on the curve specializing to the divisors constructed to compute the gonality of the graph.

Similarly, the graph $K_d'$ obtained by removing up to $d-2$ edges not forming a $K_3$ configuration is the dual graph of a model of a plane curve of degree $d$ with up to $d-2$ nodes, after blowing up along the nodes, see Section \ref{nodes}. The result of Coppens and Kato mentioned above states that the normalization of such curve has gonality $d-2$, the same as the graph $K_d'$. Also in this case, it is possible to lift the divisors of degree $d-2$ and rank $1$ on the graph.

\vspace{\baselineskip}
We show that the first difference between the gonality of these graphs and curves occurs when removing $d-1$ edges. In fact it is possible to obtain a connected graph $K_d'$ of gonality $d-3$ by removing $d-1$ edges not forming a $K_3$ configuration. For the result of Coppens and Kato mentioned above, it is no longer possible to use plane singular model to solve this lifting problem. We will explain how to find a curve $X$ such that its skeleton is $K_{d}'$ and such that the divisor of degree $d-3$ and rank one lifts, constructing a tropical morphism $\varphi: K_d' \rightarrow \mathbb{TP}^1$ of degree $d-3$ and lifting it to a harmonic morphism of triangulated punctured curves, using results from \cite{ABBR1, ABBR2}. 

\vspace{\baselineskip}
These results have also an interpretation with respect to the moduli space of curves. The dual weighted graph of a stable curve is defined as a pair $(G,w)$, where $G$ is the dual graph of the curve and $w$ is the weight function $w: V(G) \rightarrow \mathbb{Z}_{\geq 0}$, that associates to every vertex the genus of the irreducible component corresponding to it. The Deligne-Mumford space $\overline{M}_g$ of stable curves of genus $g \geq 2$ has a stratification in loci $M_{(G,w)}$ parametrizing curves with the same dual weighted graph $(G,w)$, see \cite{ACG}, 
\[
\overline{M}_g\setminus M_g = \bigsqcup_{(G,w), \ \textrm{genus} \ g} M_{(G,w)}.
\]
Inside $M_g$ there is the locus $M_{g,k}$ of curves having the same gonality $k$. If we consider $g$ as the genus of a plane curve of degree $d$ with $i$ nodes, that is $g = \frac{(d-1)(d-2)}{2} - i$, and $K_d'$ as finite graph with weight function equal to zero, our results can be interpreted in terms of moduli spaces saying that 
\[
M_{K'_d} \cap \overline{M}_{g,d-3} = \emptyset,
\]
where $K_d'$ is obtained by removing $i < d-1$ edges not forming a $K_3$ configuration. Instead, when $i=d-1$, we can construct graphs $K_d''$, for which
\[
M_{K''_d} \cap \overline{M}_{g,d-3} \not = \emptyset.
\]

\medskip

The paper has the following outline. In Section \ref{mg} we give a short introduction on the theory of divisors on metric graphs, recalling main definitions and results of \cite{BN, MZ, GK}. 

The results presented in Section \ref{2} are combinatorial and deal with the computation of the gonality of graphs $K_d$ with edges removed. 

\vspace{\baselineskip}
\noindent {\bf Theorem \ref{completegen}.} Let $K_d' = K_d \setminus \{ e_1, e_2, \dots, e_i \}$ with $d\geq 3$ and $1\leq i \leq d-2$. Then $K_d'$ has gonality $d-h$ if and only if the edges we remove contain a $K_h$ but not a $K_{h+1}$ configuration. 

\vspace{0.3cm}
For $h=2$ the upper bound on the number of removed edges is sharp.

\vspace{0.3cm}
\noindent {\bf Theorem \ref{stable}.} It is possible to obtain a connected metric graph $K_d'$ of gonality $d-3$ by omitting $d-1$ edges from the complete graph $K_d$ not containing a $K_3$ configuration. 
\vspace{0.3cm}

The interpretation of divisors as chip configurations will help us in the combinatorics of the proofs and the main tools we will use are the definition of reduced divisors and their properties, see \cite{HKN, Luo}. 

Baker's Specialization map and inequality \cite{Bak} are presented in Section \ref{sp}, together with a short account on Berkovich skeleta. 

In Section \ref{pc} we explain the lifting problem of $K_d' = K_d \setminus K_h$ and of $K_d' = K_d \setminus \{e_1, e_2, \dots, e_i\}$ with $i \leq d-2$. 
 
Finally, Section \ref{hm} is devoted to the geometric interpretation of Theorem~\ref{stable}. Finite harmonic morphisms between metric graphs, their lifts to finite harmonic morphisms of curves are introduced, as in \cite{ABBR1, ABBR2}.

\begin{section}{Metric graphs and divisors} \label{mg}
The definitions and results introduced in this section are well known in the framework of tropical geometry. 
The notion of linear systems was firstly introduced on finite graphs by Baker and Norine in \cite{BN} and then on metric graphs by Mikhalkin and Zharkov in \cite{MZ} and by Gathmann and Kerber in \cite{GK}.
\begin{subsection}{Metric graphs}
A {\sl topological graph $\Gamma$} is topological realization of a finite graph. More precisely, it is a compact, connected topological space such that for every $p \in \Gamma$ there exists a neighborhood $U_p$ of $p$ homeomorphic to a \lq \lq star with $r$ branches", for a certain $r \in \mathbb{Z}_{\geq 0}$, that is the union of the segments in $\mathbb{R}^2$ connecting the origin $(0,0)$ with $r$ points, every two of which lie in different lines through the origin. The number $r$ is unique for every point of the graph, it is called the {\sl valence} and indicated with $\textrm{val}(p)$. The integer $r$ is different from $2$ only for a finite number of points in $\Gamma$. The points for which $r \not =2$ are called {\sl essential vertices}.

\begin{defi} A {\sl metric graph} is a topological graph equipped with a complete inner metric on $\Gamma \setminus V_{\infty}(\Gamma)$, where $V_{\infty}(\Gamma)$ is a set of $1$-valent vertices called {\sl infinite vertices}. We can extend the metric also to the infinite vertices, stating that their distance from all the other points of the graph is infinite.  

A vertex set $V(\Gamma)$ is a finite subset of points of $\Gamma$ containing the essential vertices. The closures of the connected components of $\Gamma \setminus V(\Gamma)$ are called {\sl edges}. We indicate with $E(\Gamma)$ the set of edges associated to $V(\Gamma)$. The edges adjacent to infinite vertices are called {\sl infinite edges}. Using the metric to every edge $e \in E(\Gamma)$ it is possible to associate a length $l(e) \in \mathbb{R} \cup \{ \infty\}$. We have $l(e) = \infty$ if and only if $e$ is an infinite edge. 

Given a point $p \in \Gamma$, we define the set of {\sl tangent directions $T_p(\Gamma)$} at $p$ as the set of connected components of $U_p \setminus \{p\}$, where $U_p$ is a neighborhood of $p$ as defined above. We make this definition independent from a choice of another neighborhood $U_p'$, by identifying the components where the intersection of $U_p \setminus \{p\}$ and $U_p' \setminus \{p\}$ is not empty. The set $T_p(\Gamma)$ consists of~$\textrm{val}(p)$ elements.  

Given a vertex set of a metric graph $\Gamma$, its genus $g$ is defined as the first Betti number, 
\[
g = |E(\Gamma)| - |V(\Gamma)| +1. 
\]
This number is independent of the choice of a vertex set. 
\end{defi}

\begin{defi}
Let $\Lambda$ be a non-trivial subgroup of $\mathbb{R}$. A metric graph $\Gamma$ is a {\sl $\Lambda$-metric graph} if the distance between any two finite essential vertices is in $\Lambda$. A {\sl $\Lambda$-point} is a point in $\Gamma$ whose distance from any finite essential vertex is in~$\Lambda$. A vertex set for a $\Lambda$-metric graph is a vertex set containing only $\Lambda$-points of $\Gamma$. In this case, the lengths of the edges lie in $\Lambda \cup \{ \infty\}$.
\end{defi}
\end{subsection}

\begin{subsection}{Divisors and complete linear systems}
\begin{defi} Let $\Gamma$ be a metric graph. A {\sl divisor} on it is an element of the free abelian group $\textrm{Div}(\Gamma)$ on the points of the graph. Any divisor can be represented in a unique way as a finite formal combination of points in~$\Gamma$ with integer coefficients:
\[
D = \sum_{p \in \Gamma} a_p \, (p), \ \ \textrm{with} \ a_p \in \mathbb{Z}. 
\]
The {\sl degree} of a divisor $D$ is the sum of its coefficients $\textrm{deg}(D) = \sum_{p \in \Gamma} a_p$. If $a_p \geq 0$ for every $p \in \Gamma$, the divisor is said to be {\sl effective}. The {\sl support of a divisor $D$} is the set of points of $\Gamma$ such that $a_p \not =0$ and it is indicated with $\textrm{supp}(D)$.
\end{defi}

\begin{defi}
A {\sl rational function} $f$ on $\Gamma$ is a continuous piecewise linear function $f: \Gamma \rightarrow \mathbb{R}$ with integer slopes and only finitely many pieces. The {\sl principal divisor} $\textrm{div}(f)$ associated to $f$ is the divisor whose coefficient at $p$ is given by the sum of the outgoing slopes of $f$ at that point. Only for a finite number of points, the coefficients are not zero.

Two divisor $D_1$, $D_2 \in \textrm{Div}(\Gamma)$ are {\sl linearly equivalent}, $D_1 \sim D_2$, if there exists a rational function $f$ such that 
\[
D_1  - D_2 = \textrm{div}(f). 
\]
The {\sl linear system} of $D$, indicated with $|D|$, is the set of the effective divisor linearly equivalent to $D$, 
\[
|D| = \{ E \in \textrm{Div}(\Gamma)| \ E \geq 0, E \sim D\}.
\]
The {\sl rank $\textrm{rk}_{\Gamma}(D)$ of a divisor} is defined as $-1$ if $D$ is not equivalent to any effective divisor, otherwise
\[
\textrm{rk}_{\Gamma}(D)= \max \{ r \in \mathbb{Z}_{\geq 0}| \ |D-E| \not = \emptyset \ \ \  \forall \ E \in \textrm{Div}(\Gamma), \ E \geq 0, \ \textrm{deg}(E)=r\}. 
\]
\end{defi}

\begin{defi}The {\sl canonical divisor} of a metric graph $\Gamma$, indicated with~$K_{\Gamma}$, is defined as
\[
K_{\Gamma} = \sum_{p \in \Gamma} \Big(\textrm{val}(p)-2 \Big)(p).
\]
We remark that by definition of graph this is a finite sum.
\end{defi}
When no confusion is possible, we will omit the symbol $\Gamma$ in the notation of rank and canonical divisor. 

\vspace{\baselineskip}
The following analogue of the Riemann-Roch for algebraic curves holds for metric graphs, see Proposition 3.1 and Corollary 3.8 of \cite{GK} and Theorem~7.3 of \cite{MZ}:
\begin{thm}[Riemann-Roch Theorem] Let $D$ be a divisor on a metric graph $\Gamma$, then
\[
\textrm{rk}_{\Gamma}(D) - \textrm{rk}_{\Gamma}(K_{\Gamma}-D) = \textrm{deg}(D) +1 -g.
\]
\end{thm}

\end{subsection}

\begin{subsection}{Chip--firing moves} Divisors and linear equivalence of divisors can be interpreted in terms of \lq \lq configuration of chips" and \lq \lq chip-firing moves" in a more visual way. A divisor in fact is a configuration of chips on the graph, where for every point $p \in \Gamma$ there is a pile of $D(p)$ chips or $-D(p)$ anti-chips if $D(p)$ is negative.

Chips can move in the graph by chip-firing moves: if $p$ is a point of $\Gamma$ then $\textrm{val}(p)$ chips can move from it along the tangent directions with the same speed. When a chip and an anti-chip meet, they cancel out. 

Two divisors are linearly equivalent if the chip configuration corresponding to one divisor can be obtained from the configuration corresponding to the other via chip-firing moves. 

In our following arguments we will often treat divisors as configuration of chips and linear equivalence of divisors as sequence of chip-firing moves. 
\end{subsection}

\begin{subsection}{Reduced divisors and rank--determining sets} One of the main tools used in the proof of the Riemann--Roch Theorem for graphs in \cite{BN} is the definition of divisor reduced with respect to a vertex. In \cite{HKN} the analogous definition is given for metric graphs. 

\begin{defi}
Let $\Gamma$ be a metric graph and $X$ be a closed connected subset of $\Gamma$. Given $p \in \partial X$, the {\sl outgoing degree of $X$ at $p$} is defined as the maximum number of internally disjoint segments in $\Gamma \setminus X$ with an open end in $p$, so the number of tangent directions leaving $X$ at $p$. Let $D$ be a divisor on $\Gamma$. A boundary point $p \in \partial X$ is {\sl saturated with respect to $X$ and $D$} if $D(p) \geq \textrm{outdeg}_X(p)$, and {\sl non-saturated} otherwise. 
A divisor $D$ is {\sl $p$-reduced} if it is effective in $\Gamma \setminus \{p\}$ and each closed connected subset $X \subseteq \Gamma \setminus \{ p \}$ contains a non-saturated boundary point.
\end{defi}

\begin{thm}[Proposition 7.5 in \cite{MZ}]
Let $D$ be a divisor on a metric graph $\Gamma$. For every point $p \in \Gamma$ there exists a unique $p$-reduced divisor linearly equivalent to $D$. 
\end{thm}

Two properties of reduced divisors will be particularly relevant for the combinatorial results presented in this paper, and we briefly recall them.  

\begin{lem} If $D$ is a $p$-reduced divisor, then $D(p) \geq E(p)$ for every $E \sim D$, $E$ effective in $\Gamma \setminus \{p\}$. 
\end{lem}

Let $S \subseteq \textrm{supp}(D)$, following \cite{Luo}, we use the notation $\mathcal{U}_{S,p}$ to indicate the connected component of $\Gamma \setminus S$ containing $p$. 

\begin{lem}[Lemma 2.4 in \cite{Luo}] \label{lemma} Let $\Gamma$ be a metric graph, $p$ be a point of $\Gamma$ and $D$ and effective divisor in $\Gamma$. Then $D$ is $p$-reduced if and only if for any subset $S$ of $\textrm{supp}(D) \setminus \{p\}$, the set $\mathcal{U}^c_{S,p}$ contains a non-saturated boundary point with respect to $D$ and $\mathcal{U}^c_{S,p}$. 
\end{lem}
Let $D$ be a divisor such that $D(q) \geq 0$ for every $q \in \Gamma \setminus \{p\}$. It is possible to check whether $D$ is $p$-reduced using the so called \lq\lq {\sl burning algorithm}" (see \cite[\S 1]{AB}). We consider the divisors as set of firefighters, so at every point $q$ of the graph there are $D(q)$ firefighters. We imagine to start a fire at the point $p$. The fire will spread along the graph and it can be stopped only if at a point $q$ there are more firefighters than directions from which the fire is coming. If it is not possible to stop the fire and the whole graph burns down, then the divisor is $p$-reduced. 

\vspace{\baselineskip}
In the just mentioned paper, Luo also gave the following definition:
\begin{defi}
A {\sl rank--determining set} is a non-empty subset $A$ of $\Gamma$ such that for every divisor $D$
\[
\textrm{rk}(D) = \max \left\{   r \in \mathbb{Z}_{\geq 0}\, \Big|  \begin{array}{l}  \ |D-E| \not = \emptyset, \ \forall \ E \in \textrm{Div}(\Gamma), \ \textrm{supp}(E) \subseteq A, \\
 \ E \geq 0, \ \textrm{deg}(E)=r \\
\end{array}
\right\}.
\] 
\end{defi}
He gave a complete characterization of rank--determining sets for a graph and in particular he proved that there exists always a rank--determining set consisting of $g+1$ points. For our purposes the following results suffices. 
\begin{thm} \label{rankdet} Let $V(\Gamma)$ be a vertex set of $\Gamma$. Then $V(\Gamma)$ is a rank-determining set.
\end{thm}
\end{subsection}
\end{section}

\begin{section}{Combinatorial results} \label{2} In this section we focus on combinatorial results concerning the gonality of metric graphs.  Let $\Gamma$ be a metric graph. As in the theory of divisors on algebraic curves we use the notation $g^r_d$ to indicate a linear system of degree~$d$ and rank $r$. 
\begin{defi} The {\sl gonality} $\textrm{gon}(\Gamma)$ of a graph $\Gamma$ is the smallest integer $d$ for which there exists a $g^1_d$. 
\end{defi}
With the notation $K_d$ we indicate the complete metric graph on $d$ vertices where the edges have an arbitrary length; we fix as vertex set the essential vertices. 

\begin{exam} \label{exgon} The complete graph $K_d$ has gonality $d-1$. In fact, for every subset of $d-1$ vertices $J = \{x_1, x_2, \dots, x_{d-1}\}$, the divisor $D_J = \sum_{i=1}^{d-1} (x_i)$ has rank~$1$ and degree $d-1$.   
\end{exam}   
\begin{subsection}{A lower bound} 

Sofie Burggraeve, a student of Filip Cools, in her unpublished master thesis  at KU Leuven with title \lq \lq Tropical Geometry, Linear Systems on Metric Graphs" proved that if a metric graph contains the complete graph $K_d$ as a subgraph, then its gonality is at least $d-1$. She also proved a generalized statement, namely that if the metric graph contains $d$ distinct vertices such that every pair of these $d$ vertices is linked with a path and such that all these $\binom{d}{2}$ paths are disjoint except at their boundary points, its gonality is at least $d-1$. 

The following result, similar but a bit stronger than the one just stated, also gives a lower bound for the gonality of a metric graph $\Gamma$ in terms of the number of disjoint paths that can be found between vertices of the graph. 

\begin{prop} Let $\Gamma$ be a metric graph and let $\{v_1, v_2, \dots, v_n\}$ be a subset of $V(\Gamma)$ such that for every $i$ the vertex $v_i$ is connected to every $v_j$ with $j \not = i$ by $n-1$ disjoint paths except for the end points. Then the gonality of $\Gamma$ is at least $n-1$.  
\end{prop}
\begin{proof} Let $D$ be a divisor of degree $n-2$, we prove that it cannot have rank~$1$. We reduce it with respect to $v_1$ and suppose that $D(v_1) \geq 1$, otherwise we can conclude. Since the degree is $n-2$ there is at least a vertex $v_j$ such that $D(v_j) = 0$ and one of the paths that connects it with $v_1$ does not contain any point of the support of $D$. Using Lemma~\ref{lemma}, we prove  that $D$ is also $v_j$-reduced. Let $S \subseteq \textrm{supp}(D) \setminus \{v_j\}$, we recall that with the notation $\mathcal{U}^c_{S,v}$we indicate the complement of the connected component of $\Gamma \setminus S$ containing $v$. We distinguish two cases: 
\begin{enumerate}
\item $v_1 \not \in S$. \\
In this case $\mathcal{U}^c_{v_1, S} = \mathcal{U}^c_{v_j, S}$ and since $D$ is $v_1$-reduced, by Lemma \ref{lemma} we can find a non-saturated element. 
\item $v_1 \in S$. \\
We prove that $v_1$ is the non-saturated boundary point for $\mathcal{U}^c_{v_j, S}$. In fact $\textrm{val}(v_1) \geq n-1$ since at least $n-1$ disjoint paths are leaving from $v_1$, namely the ones connecting it with the vertices $\{v_2, \dots, v_n\}$. The degree of the divisor is $n-2$, so at most $n-2-D(v_1)$ of these edges can contain chips in points different from $v_1$. Therefore 
\begin{eqnarray*}
\textrm{outdeg}_{\mathcal{U}^c_{v_j, S}}(v_1) &\geq & \textrm{val}(v_1) - (n-2 - D(v_1)) \\
& \geq & n -1 - n + 2 + D(v_1) > D(v_1), \\
\end{eqnarray*}
and $v_1$ is non-saturated. 
\end{enumerate}
The divisor is also $v_j$-reduced and since $D(v_j) = 0$ it cannot have rank $1$, so the gonality of $\Gamma$ has to be at least $n-1$.
\end{proof}
\begin{cor} \label{complete} Let $\Gamma$ be a metric graph containing $K_d$ as a subgraph. The gonality of $\Gamma$ is at least $d-1$. 
\end{cor}

\end{subsection}
\begin{subsection}{Omitting $K_h$ configurations}
Now we study the gonality of a connected metric graph $K_d'$ obtained by removing edges from $K_d$. With the expression \lq \lq omitting a $K_h$" and the notation $K_d \setminus K_h$ we mean removing the edges that form the complete graph on $h$ vertices for $h < d$. 

\begin{thm} \label{hconf}
Let $K_d' = K_d \setminus K_h$ with $d \geq 3$. The gonality of $K_d'$ is $d-h$.
\end{thm}

\begin{proof} We indicate with $v_1, v_2, \cdots, v_h$ the essential vertices of $K_h$. The divisor $D = (v_{h+1}) + (v_{h+2}) + \cdots + (v_d)$ has degree $d-h$ and rank $1$. In fact, since the vertices $v_1, v_2, \dots, v_h$ are completely disconnected with respect to each other, the chips can move towards $v_1$ and form an effective divisor linearly equivalent to $D$ with chips in $v_1$. The same can be done for all the other vertices $v_2, v_3, \dots, v_h$. By Theorem \ref{rankdet} the divisor has rank at least $1$, so $K_d'$ has gonality at most $d-h$. 
The vertices $v_{h+1}, \dots, v_d$ and a vertex $v_i \in \{v_1, v_2, \dots, v_{h} \}$ form a complete graph $K_{d-h+1}$, 
so it follows from Corollary \ref{complete} that the gonality has to be at least $d-h$. 
\end{proof}
\begin{exam} \label{esempio} Let $d=8$. We remove $6$ edges forming a $K_4$ on the vertices $v_1, v_2, v_3$ and $v_4$. In the left-hand side graph of Figure \ref{FigEsem1} the removed edges are dashed. The divisor $D = (v_5) + (v_6) + (v_7) + (v_8)$ has degree $4$ and rank at least $1$, since we can move the chips along the edges connecting the vertices with $v_1$; at least one chip will reach the vertex $v_1$. The same can be done for the other vertices. Therefore the gonality of $K_8'= K_8 \setminus K_4$ is at most $4$. The vertices $v_4, v_5, v_6, v_7, v_8$ form a $K_5$ contained in the graph $K_8'$, so the gonality is actually $4$. In the right-hand side figure the edges forming the $K_5$ are dotted. 
\begin{figure}[h]
\includegraphics[width=40mm,height=40mm]{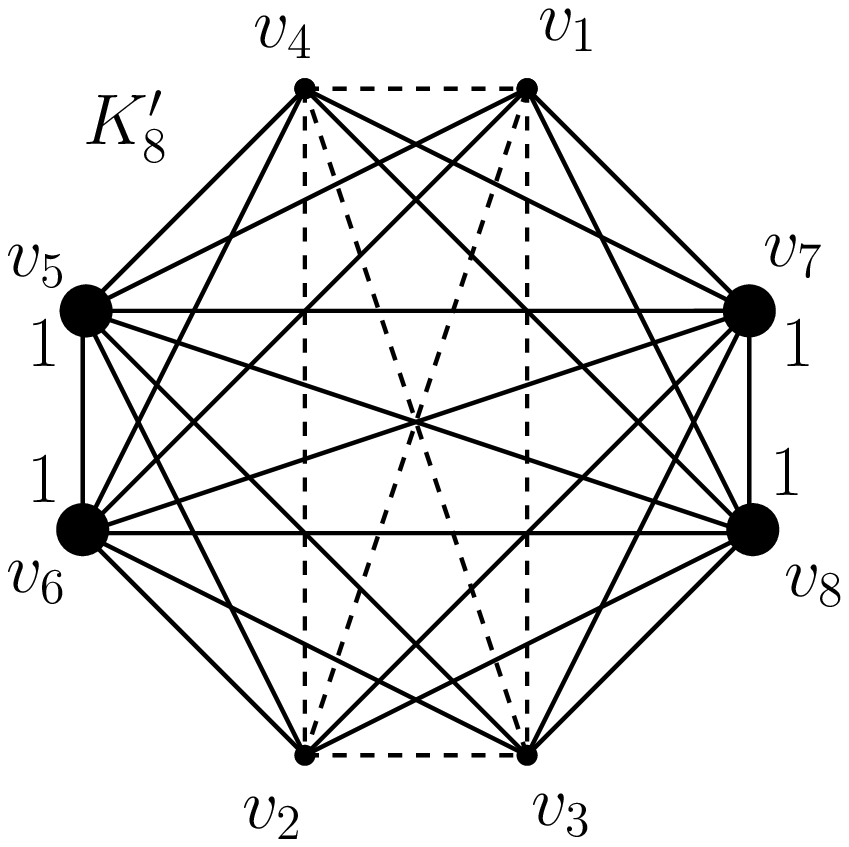}~a)
\qquad \qquad
\includegraphics[width=40mm,height=40mm]{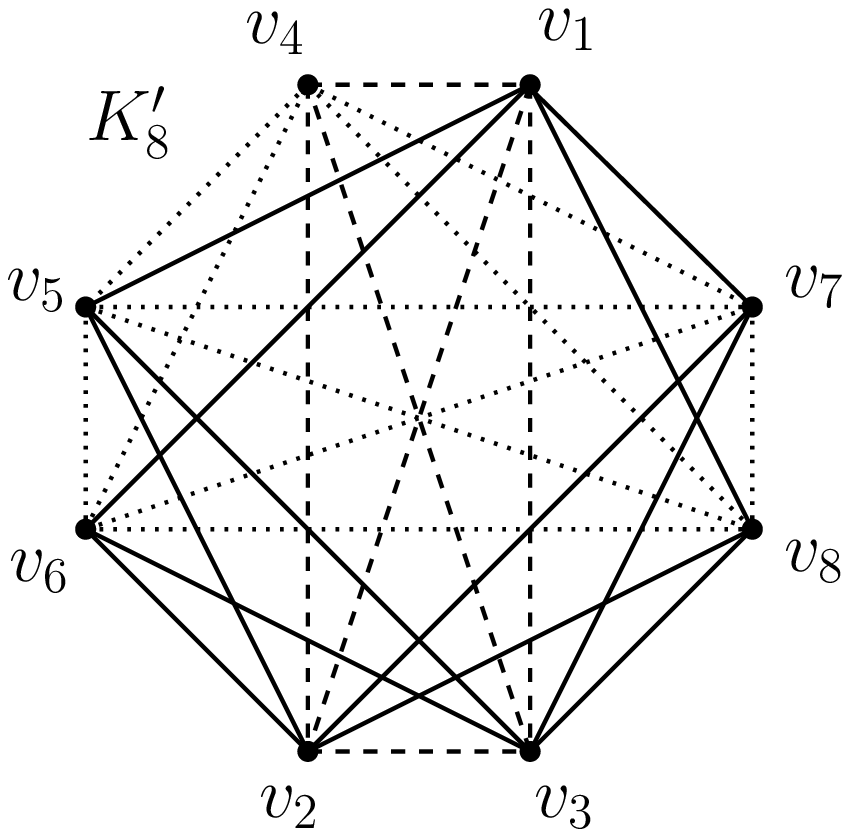}~b)
\caption{The graph $K_8'$ of Example \ref{esempio}. The omitted edges forming a $K_4$ are dashed. In  \textbf{a} it is depicted the divisor $D = (v_5) + (v_6) + (v_7) + (v_8)$ of degree $4$ and rank at least $1$.  In \textbf{b} the edges forming a $K_5$ are dotted to show that the gonality is actually $4$.}
\label{FigEsem1}
\end{figure}
\end{exam}

\begin{thm} \label{duecomplete} Let $K_d' = K_d \setminus \{K_m, K_n\}$, connected. The gonality of $K_d'$ is $d - \textrm{max}\{m,n\}$. 
\end{thm}
\begin{proof} Without loss of generality we assume $n \geq m$. We suppose that there is at least one vertex $\overline{v}$ of $K_m$ not in $K_n$, otherwise by Theorem \ref{hconf} we can conclude. We consider a divisor $D$ of degree $d-n-1$ and we show that it cannot 
have rank $1$. We reduce it with respect to $\overline{v}$ and suppose $D(\overline{v}) \geq 1$. Since the degree is $d-n-1$, it means that there is at least one vertex $w$ in~$K_d \setminus K_n$ without chips. We show that $D$ is also $w$-reduced. We will show that for any subset $S$ of $\textrm{supp}(D) \setminus \{w\}$, the set $\mathcal{U}^c_{S,w}$ contains a non-saturated boundary point with respect to $D$ and $\mathcal{U}^c_{S,w}$. 
There are at least $d-m$ disjoint paths inside $K_d'$ connecting $\overline{v}$ and $w$. We can have chips in at most $d-n-2$ of them, therefore if $\overline{v} \not \in S$, then $\overline{v} \in \mathcal{U}_{w,S}$ and $\mathcal{U}^c_{w, S} = \mathcal{U}^c_{\overline{v}, S}$. Since $D$ is $\overline{v}$-reduced then we can find a non-saturated element. 

If $\overline{v} \in S$, then $\overline{v}$ is the non-saturated boundary point for $\mathcal{U}^c_{w, S}$. In fact $\textrm{val}(\overline{v}) =d-m$, since $\overline{v} \in K_m$. The degree of the divisor is $d-n-1$, so at most $d-n-1-D(\overline{v})$ of these edges can contain chips in points different from $\overline{v}$. Therefore we have 
\begin{eqnarray*}
\textrm{outdeg}_{\mathcal{U}^c_{w, S}}(\overline{v}) &\geq & \textrm{val}(\overline{v}) - (d-n-1 - D(\overline{v})) \\
& \geq & d - m - d +n +1 + D(\overline{v}) > D(\overline{v}), \\
\end{eqnarray*}
and $\overline{v}$ is non-saturated. 
\end{proof}
The gonality of the complete bipartite metric graph follows immediately from Theorem \ref{duecomplete}. 
\begin{cor} \label{bipartite} Let $K_{m+n}$ be a complete metric graph with $m,n \geq 2$. If we remove disjoint $K_m$ and $K_n$ we obtain the complete bipartite graph $K_{m,n}$. 
By the previous theorem we have 
\[
\textrm{gon}(K_{m,n}) = m+n - \max\{m,n\} = \min\{m,n\}. 
\]
\end{cor}
\end{subsection}

\begin{subsection}{Omitting general configurations}
Theorem \ref{hconf} can be generalized to compute the gonality of $K_d' = K_d \setminus \{e_1, e_2, \dots, e_i\}$, with no hypothesis on the configuration formed by the removed edges. The case of $i = 1, 2$ has already been considered by Sofie Burggraeve in her master's thesis, where she proved that for 
both these values the gonality is $d-2$. We will assume $d \geq 3$ and consider the case $1 \leq i \leq d-2$; 
with this upper bound the graph $K_d' =  K_d \setminus \{ e_1, e_2, \dots, e_{i} \}$ will always be connected. The upper bound on the number of removed edges will be further justified in Section \ref{sharpness}. 

\begin{thm} \label{completegen} Let $K_d' = K_d \setminus \{ e_1, e_2, \dots, e_i \}$ with $d\geq 3$ and $1\leq i \leq d-2$. Then $K_d'$ has gonality $d-h$ if and only if the edges we remove contain a $K_h$ but not a $K_{h+1}$ configuration. 
\end{thm}

\begin{proof} We set $i  = d-2$, but all the arguments work also for $1 \leq i \leq d-2$. First of all we notice that not every vertex of $K_d'$ can have valency smaller or equal than $d-3$. In fact $K_d'$ has 
\[
\frac{d(d-1)}{2} - d + 2 = \frac{d(d-3)}{2} + 2 
\]
edges, while if all the vertices had valency smaller or equal than $d-3$, the graph would have at most $\frac{d(d-3)}{2}$ edges. So there must be at least one vertex with valency at least $d-2$.  

\vspace{\baselineskip}
If $\frac{h(h-1)}{2}$ edges of the omitted ones form a $K_h$ with $h  < d$, we can argue as we did in the proof of Theorem \ref{hconf} to show that gon$(K_d') \leq d-h$. The same cannot be done for the opposite inequality: in general, after removing more then $\frac{h(h-1)}{2}$ edges, it might not be possible to find a complete $K_{d-h+1}$ as a subgraph. Therefore the argument in the second part of that proof cannot be applied. 

To prove our result we consider a divisor $D$ of degree $d-h-1$ and we show that it cannot have rank $1$. Let $\overline{v}$ be a vertex with valency bigger or equal than $d-2$. In the situation when val$(\overline{v}) = d-2$, we indicate with $w$ the vertex not connected with $\overline{v}$. We reduce $D$ with respect to $\overline{v}$ and we use the notation $D$ also for the reduced divisor. If $D(\overline{v}) \leq 0$, then the rank is not~$1$, so we suppose $D(\overline{v}) \geq 1$. It is important to remark that, since the divisor is reduced, at most one chip can lie in the interior of every edge. The degree is $d-h-1$, so at least $h+1$ vertices $\{v_1, v_2, \dots, v_{h+1} \}$, the edges between them and the edges connecting them with $\overline{v}$ will be without chips. Moreover, since we are not omitting a $K_{h+1}$ at least two of them will be connected by an edge: without loss of generality we suppose that these vertices are $v_1$ and~$v_2$. We show that the divisor is reduced also with respect to $v_1$, so the rank cannot be $1$. 

\vspace{\baselineskip}
By definition of reduced divisor, it is necessary to show that each closed connected subset $X \subseteq K_d' \setminus \{ v_1 \}$ contains a 
non-saturated boundary point, that is a point $p \in \partial X$ such that $D(p) < \textrm{outdeg}_X(p)$. If $X \subseteq K_d' \setminus \{v_1, \overline{v} \}$, then there is such $p$ because the divisor is reduced respect to $\overline{v}$. 
So assume $\overline{v} \in X$. Boundary points $p$ such that $p \not \in \textrm{supp}(D)$ are non-saturated. So, assume that all the boundary points of $X$ are points of the support of $D$. A boundary point lying in the interior of an edge is saturated. If $v_2 \in \partial X$ then $v_2$ is non-saturated because $\textrm{outdeg}_X(v_2) \geq 1 > D(v_2) = 0$. 

\vspace{\baselineskip}
These preliminary remarks lead us to consider $X \subseteq K_d'  \setminus \{v_1 \}$ such that $v_2 \not \in \partial X$, $\overline{v} \in X$, with as boundary points only vertices or interior points of edges lying in the support of $D$. This in particular implies that $v_2 \not \in X$. Moreover $\overline{v} \in \partial X$, since the edges connecting it with $v_1$ and $v_2$ do not contain chips. 

We denote with  $A$ the set $X \cap V(K'_d)$ and with $j$ the number of its element, so $j= |A|$, with $1 \leq j \leq d-2$. Suppose that we cannot find a non-saturated element in $A \cap \partial X$ and $\textrm{outdeg}_X(v) = 0$ for $v \in A$, $v \not \in \partial X$. It follows that also every element of $\partial X$ is saturated. We will deduce from these assumptions which edges are necessary to be removed from $K_d$ to obtain a suitable $K_d'$ and we will show that they are more than~$d-2$.

For every vertex $v$ we define
\[
X_{v,A} = \{p \in \partial X \ \textrm{such that} \ p \in \textrm{int}(vu), \ \textrm{with} \ u \in V(K_d') \setminus A\}.
\]
By the hypothesis on $X$, the set $X_{v,A}$ is contained in $\textrm{supp}(D)$.

All the elements in $A$ are saturated, thus $D(\overline{v}) \geq \textrm{outdeg}_X(\overline{v})$. 
The following inequality holds: 
\[\textrm{outdeg}_X(\overline{v}) + |X_{\overline{v},A}| \geq 
\begin{cases} d-j &  \textrm{if val}(\overline{v}) = d-1 \\ 
&  \textrm{or val}(\overline{v}) = d-2 \ \textrm{and} \ w \in A \\ 
d- j -1 &  \textrm{if val}(\overline{v}) = d-2 \ \textrm{and} \ w \not \in A.
\end{cases}
\]
  This inequality is an equality when all the edges connecting vertices contained in $X$ are also in $X$. If this condition does not hold, then the number of edges removed is actually bigger: for a vertex $v \in \partial X$, such that $vu$ is not in $X$, with $u \in \partial X$, in order to make it saturated, it is necessary to omit one edge more respect to the previous count. 

From the inequality above it follows that
\begin{numcases}{D(\overline{v}) + |X_{\overline{v},A}| \geq} \label{caso1}
d-j &  if $\textrm{val}(\overline{v}) = d-1$ \\ \nonumber \label{caso2}
& or $\textrm{val}(\overline{v}) = d-2 \ \textrm{and} \ w \in A$ \\
d- j -1 & if $ w \not \in A$.
\end{numcases} 

The degree of $D$ is $d-h-1$, so in the hypothesis characterizing the equation (\ref{caso1}), $j$ needs to satisfy the inequality $j \geq h+1$, otherwise $\overline{v}$ is non-saturated. Similarly, $j \geq h$ in the hypothesis of (\ref{caso2}). Hence, in the first case we can have at most $d-h-1-d+j = j -h-1$ chips in $X \setminus  (\{\overline{v}\} \cup X_{\overline{v},A})$ and $j-h$ in the second one. 

We focus on the first one. Since no elements in $A$ are non-saturated, it means that all the vertices such that $D(v) + |X_{v,A}|= 0$ are not boundary points, so that the edges connecting them with the vertices not in $A$ are omitted. Therefore, for every such vertex, $d-j$ edges are removed since~$d-j$ elements are not in $A$. Moreover also for the vertices $v$ such that $D(v) + |X_{v,A}| = 1$ at least $d- j -1$ edges must be removed since we need to have $D(v) \geq \textrm{outdeg}_X(v)$. In general for a vertex $v$ such that $D(v) + |X_{v,A}| = k$ with $0 \leq k \leq j-h-1$ at least $\max \{0,d-j-k\}$ edges have to be removed for having $D(v) \geq \textrm{outdeg}_X(v)$. The edges omitted are $\max \{0,d-j-k\}$ of the ones connecting $v$ with the vertices not in $A$ and that do not contain points of $X_{v,A}$. Let $l_k$ be the number of vertices $v$ in $A \setminus \{\overline{v}\}$ with $D(v) + |X_{v,A}| = k$. We have 
\[
\sum_{k=0}^{j-h-1} l_k = j-1, \ \ \ \ \textrm{and}
\]
\[
 \sum_{k=0}^{j-h-1} k l_k \leq d - h-1 - D(\overline{v}) - |X_{\overline{v},A}| \leq d- h -1 - (d-j) = j-h-1. 
\]
The number of edges removed if val$(\overline{v}) = d-1$ is at least
\begin{align}
\sum_{k=0}^{j-h-1} l_k \max \{0,d-j-k\}  & \geq \sum_{k=0}^{j-h-1} l_k (d-j-k) \nonumber \\
& = (d-j) \sum_{k=0}^{j-h-1} l_k  - \sum_{k=0}^{j-h-1} l_k k  \nonumber\\
& \geq  (j-1)(d-j) - (j-h-1). \label{calc1}
\end{align}
If val$(\overline{v}) = d-2$ and $w \in A$ we need to count also the omitted edge $\overline{v}w$, therefore the number of edges removed is at least 
\begin{equation}
\sum_{k=0}^{j-h-1} l_k \max \{0,d-j-k\} +1 \geq (j-1)(d-j) - (j-h-1) +1. \label{calc2}
\end{equation}
Remember that $h+1 \leq j \leq d-2$. It is enough to study the two local minima for (\ref{calc1}): they are achieved for $j =h+1$ and $j = d-2$ and the value is equal to respectively $h(d-h-1)$ and $d+h-3$. This values are always strictly bigger than $d-2$. Indeed, $h(d-h-1) - (d-2) = (h-1)(d-h-2) \leq 0$ if and only if $d \leq h+2$, but this is not possible. In fact if  $d=4$ and $h=2$, then $j$ needs to satisfy $3 \leq j \leq 2$ and if $d=5$ and $h=3$, $j$ needs to satisfy $4 \leq j \leq 3$. If $d \geq 5$ the number of edges of $K_h=K_{d-2}$ is bigger than $d-2$. So if $j = h+1$ the number of edges omitted is strictly bigger than $d-2$. Moreover $d + h - 3 - (d-2) = h -1 > 0$, since $h \geq 2$. So also for $j = d-2$ the number of omitted edges is strictly bigger than $d-2$. 

\vspace{\baselineskip}
A similar computation works when val$(\overline{v}) = d-2$ and $w \not \in A$ with $h \leq j \leq d-2$. Again, let $l_k$ be the number 
of vertices $v$ of $A \setminus \{\overline{v}\}$ such that $D(v) + |X_{v,A}| = k$ with $0 \leq k \leq j-h$. We have 
\[
\sum_{k=0}^{j-h} l_k = j-1, \ \ \ \ \textrm{and}
\]
\[ \sum_{k=0}^{j-h} k l_k \leq d - h -1 - D(\overline{v}) - |X_{v,A}| \leq d-h-1 - (d-j -1) = j-h. 
\]
So, remembering the omitted edge $w\overline{v}$, since $\textrm{val}(\overline{v}) = d-2$, the number of edges removed is 
\begin{align} 
\sum_{k=0}^{j-h-1} l_k \max \{0,d-j-k\} +1 & \geq \sum_{k=0}^{j-h} l_k (d-j -k) +1 \nonumber \\ 
&= \sum_{k=0}^{j-h} l_k (d-j) - \sum_{k=0}^{j-h} l_k k  +1 \nonumber \\
& \geq  (j-1)(d-j) - (j-h-1). \label{calc3}
\end{align}
The minima are $(h-1)(d-h) +1$ and $d+h-3$. These two values are strictly bigger than $d-2$ because $(h-1)(d-h) + 1 - (d-2) = (h-2)(d-h-1)+1$ 
and $d+h-3 - (d-2) = h -1$. \\
\end{proof}
\end{subsection}

\begin{subsection}{Sharpness} \label{sharpness} For $h=2$ and $d\geq 4$ the bound $i \leq d-2$ in the previous result is sharp. 

\begin{thm} \label{stable} It is possible to obtain a connected metric graph $K_d'$ of gonality $d-3$ by omitting $d-1$ edges not containing a $K_3$ configuration from $K_d$  
\end{thm}
\begin{proof} We briefly recall the proof of Theorem \ref{completegen} for $h = 2$: given a $\overline{v}$-reduced divisor of degree $d-3$, with val$(\overline{v}) \geq d-2$, there are at least three vertices, the edges between them and the edges connecting them with $\overline{v}$ without chips and two of these vertices $v_1$ and $v_2$ are adjacent. Let $X \subseteq K_{d}' \setminus \{v_1\}$ be a closed connected subset, such that if two vertices are in $X$ then also the edge connecting them is in $X$. We indicate with $A$ the set $X \cap V(K_d')$ and with $j$ the number of its element. The set~$X$ contains only saturated elements (i.e.\ $D$ is not $v_1$-reduced) if and only if there are no boundary points $p$ such that $p \not \in \textrm{supp}(D)$,  $\overline{v} \in \partial X$, $v_2 \not \in \partial X$, and for every vertex $v$ in $A$ such that $D(v) + |X_{v,A}| = k$, we remove at least $\max\{0,d-j-k\}$ edges. The graph $K_d'$ is obtained from $K_d$ by removing at least 
\[
\begin{cases} \sum_{k=0}^{j-3} l_k \max \{0,d-j-k\} \geq (j-1)(d-j) - (j-3) \ \textrm{edges} \\
\textrm{if val}(\overline{v}) = d-1 \ \textrm{and} \ j \geq 3 \\
\sum_{k=0}^{j-3} l_k \max \{0,d-j-k\} + 1 \geq (j-1)(d-j) - (j-3) +1 \ \textrm{edges}\\
\textrm{if val}(\overline{v}) = d-2 \ \textrm{and} \ w \in A \ \textrm{and} \ j \geq 3 \\
\sum_{k=0}^{j-2} l_k \max \{0,d-j-k\} +1 \geq (j-1)(d-j) - (j-3) \ \textrm{edges} \\
\textrm{if val}(\overline{v}) = d-2, w \not \in A \ \textrm{and} \ j \geq 2.
\end{cases}
\]
where as in the proof of Theorem \ref{completegen}, we indicate with $l_k$ the number of vertices $v$ in $A$ with $D(v) + |X_{v,A}| = k$ chips. These are the inequalities (\ref{calc1}), (\ref{calc2}) and (\ref{calc3}) for $h = 2$ in the proof of Theorem~\ref{completegen} and they were proven to be strictly bigger than $d-2$. 

\vspace{\baselineskip}
We will organize the proof in three steps: 
\begin{itemize} 
\item We will explain when the number of edges removed is equal to the above right-hand sides. In particular this implies that the inequalities in (\ref{calc1}), (\ref{calc2}) and (\ref{calc3}) for $h=2$ are equalities. 
\item Then we will compute when the right-hand side numbers are equal to $d-1$. These considerations will provide us with a description of the divisor $D$. 
\item We will check if $D$ has rank $1$. 
\end{itemize}

We begin with the first step. As pointed out in the previous proof, without the hypothesis that if two vertices are in $X$, then also the edge connecting them must be contained, the numbers of edges removed are bigger than the left-hand side numbers on the inequalities above. Moreover the number of removed edges is equal to the numbers on the left-hand side if and only if for every vertex $v$ in $A \setminus \{\overline{v}\}$ such that $D(v) + |X_{v,A}| = k$ we remove exactly $\max\{0,d-j-k\}$ edges. Then if the first inequalities in (\ref{calc1}), (\ref{calc2}) and (\ref{calc3}) are equalities, for all $v \in A$, we need $k = D(v) + |X_{v,A}| \leq d-j$. Finally the last inequalities in (\ref{calc1}), (\ref{calc2}) and (\ref{calc3}) are equalities if and only if $D(\overline{v}) + |X_{\overline{v},A}| =d-j$ for (\ref{calc1}) and (\ref{calc2}) or $D(\overline{v}) + |X_{\overline{v},A}| = d-j-1$ for (\ref{calc3}) and $A$ is such that all the chips lay in the closed edges connecting vertices in $A$ with vertices not in $A$. The last conditions ensure us that in our computation all the chips are contributing and 
\[
\sum_{k=0}^{j-2} kl_k = d-2-1-D(\overline{v}) - |X_{v,A}|.
\]

The following step is computing the integer solutions of the equations 
\[
(j-1)(d-j) - (j-3) + 1 = d-1 \ \ \ \ \ \ (j-1)(d-j) - (j-3) = d-1.
\]
The first equation does not have integer solutions except for $d=4$, but by removing $3$ edges from a $K_4$ we obtain a tree, and indeed it has gonality one. 

The second one has solutions $j_1 = 2$ and $j_2 = d-2$. Therefore, there are three different situations in which it is possible to find a subset $X \subseteq V(K_d') \setminus \{v_1, v_2\}$ with only saturated points: 

\vspace{\baselineskip}
\fbox{val$(\overline{v}) = d-1$, $j= d-2$ and $D(\overline{v}) + |X_{\overline{v},A}|=2$}  
\medskip 

Therefore $A = V(K_d') \setminus \{v_1, v_2\}$ and $D(\overline{v}) = 2$. The number of removed edges is given by 
\[
(d-d+2-0) \cdot l_0 + (d-d+2-1) \cdot l_1=2l_0 +l_1.
\]
Since $D(\overline{v}) = 2$ and $j = d-2$, then $l_0 \geq 2$. 

The divisors do not have coefficients strictly bigger than $ d - j = 2$, according to the previous step. Therefore the divisors are of the form 
\begin{eqnarray*} 
D = 2(\overline{v}) &+& 0((v_1) + \cdots +(v_s)) + (E_{s+1} + \cdots + E_t)  \\
&+& ((f_{t+1}) + \cdots + (f_{d-1})).
\end{eqnarray*}
In this expression, $E_k$ is an effective divisor of degree $2$ with support in the edges $v_kv_1$ and $v_kv_2$, while $(f_k)$ is an effective divisor of degree $1$ with support in the edges $v_kv_1$ and $v_kv_2$. Moreover $s \geq 4$ and $2+2(t-s) + (d-t-1) = d-3$, so $t-2s+4=0$. 

\vspace{\baselineskip}
We study the rank of the divisors described above, in particular we prove that they do not have rank $1$. Since $s\geq 4$ there are at least two vertices in~$A \setminus \{\overline{v}\}$ such that $D(v) + |X_{v,A}| = 0$. We show that the divisors are reduced with respect to these vertices using a burning argument and firing from one of them. These vertices are connected between each other and connected to the other vertices in $A$, therefore two fires will approach every vertex in the support of the divisor. If there are vertices with coefficient $1$, they will burn and then all the vertices with coefficient $2$ will do the same, since the fire passing from the burnt vertices will approach them too. If there are no vertices with coefficient $1$, then $l_0 >2$, so more than two fires will approach every vertex. In both situations the whole graph will burn. If instead a point $p$ of the support is not a vertex, then three fires will approach $\overline{v}$, which therefore will burn and then the whole graph will do the same. 

\vspace{\baselineskip}
\fbox{val$(\overline{v}) = d-2$, $j = d-2$, $w \in \{v_1, v_2\}$ and $D(\overline{v}) =1$} 
\medskip 

As before the number of edges that we need to remove is given by 
\[
1 + (d-d+2-0)\cdot l_0 + (d-d+2-1)\cdot l_1 = 1 + 2 l_0 + l_1.
\]
Again it is equal to $d-1$ if and only if the divisor has no vertices such that $D(v) + |X_{v,A}|$ is strictly bigger than $d - j = 2$. The divisors therefore are of the form 
\begin{eqnarray*}
D = (\overline{v}) &+& 0((v_1) + \cdots +(v_s)) +  (E_{s+1} + \cdots + E_t) \\
&+& ((f_{t+1}) + \cdots + (f_{d-1})),
\end{eqnarray*}
where $E_k$ is an effective divisor of degree $2$ with support in the edges $v_kv_1$ and $v_kv_2$, while $(f_k)$ is an effective divisor of degree $1$ with support in the edges $v_kv_1$ and $v_kv_2$. Moreover $s\geq 3$ and $2(t-s) + (d-t-1) + 1 = d-3$, so $t-2s +3=0$. 

\vspace{\baselineskip}
We claim that the only one of such divisors that has rank $1$ is 
\[
D = (\overline{v}) + (v_4) + \cdots + (v_{d-1}).
\]
If a divisor has an $E_i$ then there are two vertices without chips in $A$ that are connected and we can conclude using a burning argument as in the previous item. If there is a $(f_i)$ with the chip not in $v_i$, then we can find an equivalent effective divisor $D'$ such that $v_1, v_2 \in \textrm{supp}(D')$. Anyway there is no effective divisor $D''$ equivalent to $D$ such that $v_i \in \textrm{supp}(D)$: again by firing from~$v_i$ two fires will approach $\overline{v}$, one of which passing through $v_3$, and then the whole graph will burn. 

Instead, the divisor $D$ above is such that we can move the chips towards~$v_3$ and get an equivalent divisor with chips in it and the same for $v_1$ and $v_2$. 

\vspace{\baselineskip}
We illustrate this with an example. Let $K_8'$ be the graph in Figure \ref{FigSharp1}, with every edge with the same length. The divisor $D = 5(v_3) \sim (v_7) + (v_6) + (v_5) + (v_4) + (\overline{v}) \sim  3(v_1) + 2(v_2)$ has rank $1$. 

\begin{figure}[h]
\includegraphics[width=4cm,height=4cm]{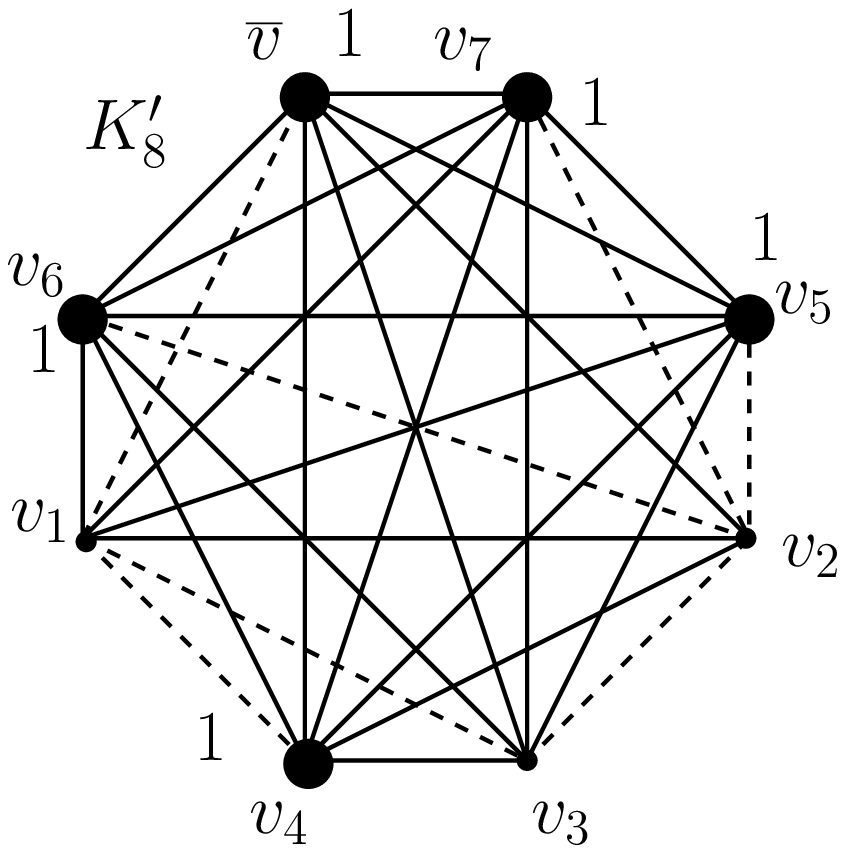}~a)
\qquad \qquad
\includegraphics[width=4cm,height=4cm]{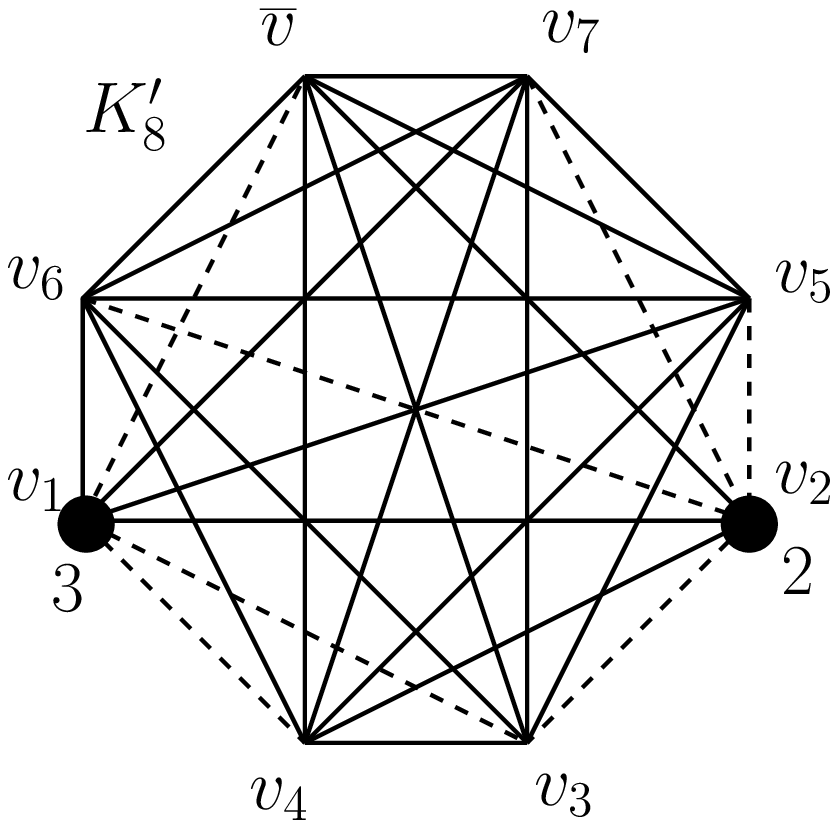}~b)
\caption{The graph $K_8'$. The omitted edges are dashed. In \textbf{a} it is depicted the divisor $(v_7) + (v_6) + (v_5) + (v_4) + (\overline{v})$ and in \textbf{b} the divisor  $3(v_1) + 2(v_2)$. They are linearly equivalent to $5(v_3)$.  }
\label{FigSharp1}
\end{figure}

\vspace{2cm}

\fbox{val$(\overline{v}) = d-2$, $j = 2$ and $A = \{\overline{v}, v\}$, with $v \not = w$} 
\medskip 

Then the divisor has to be of degree $d-3$ with support in the edges $\overline{v}u$ with $u \in K(V_d') \setminus \{v, w\}$ and at least one chip in $\overline{v}$. The number of edges that we need to remove is 
\[
1+(d-2)l_0 = 1+d-2 = d-1.
\]

The rank of the divisor depends on the graph. For example, if all the edges have the same length, we can consider the divisor $D= (d-3)(\overline{v})$. It has rank $1$, in fact 
\[
D = (d-3)(\overline{v}) \sim (d-3) (v) \sim (v_4) + \cdots + (v_{d-1}) \sim (d-3)(w). 
\]
With arbitrary lengths of the edges this does not hold. 

\vspace{\baselineskip}
Again we illustrate it with an example, see Figure \ref{FigSharp2}.

\begin{figure}[h]
\includegraphics[width=3.5cm,height=3.5cm]{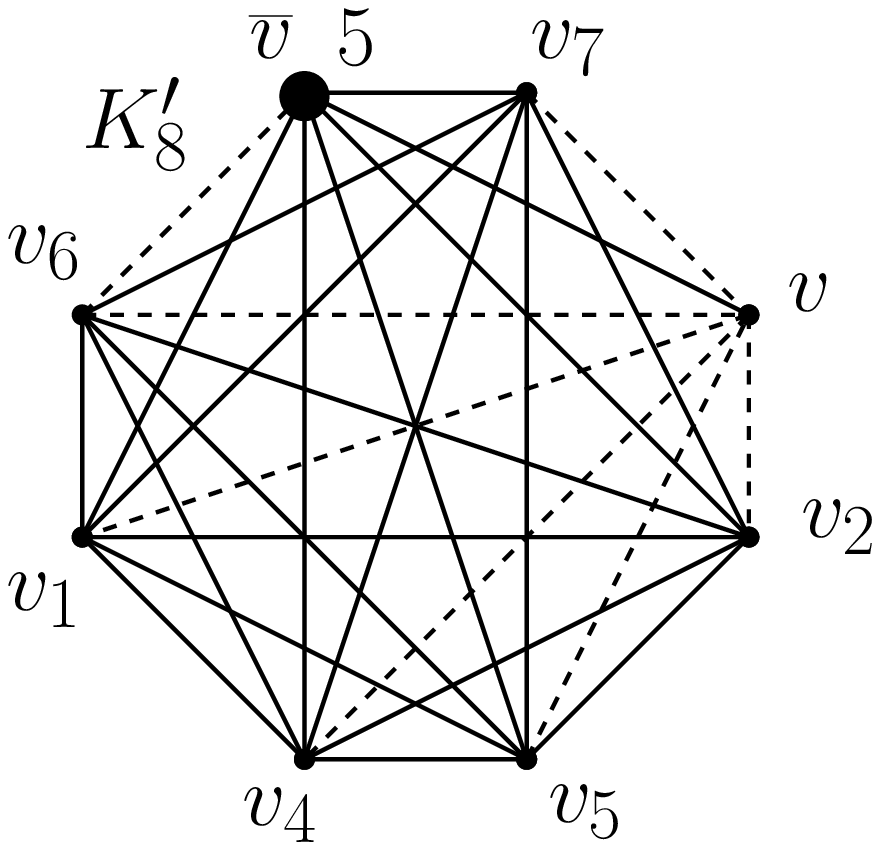}~a)
\hfill
\includegraphics[width=3.5cm,height=3.5cm]{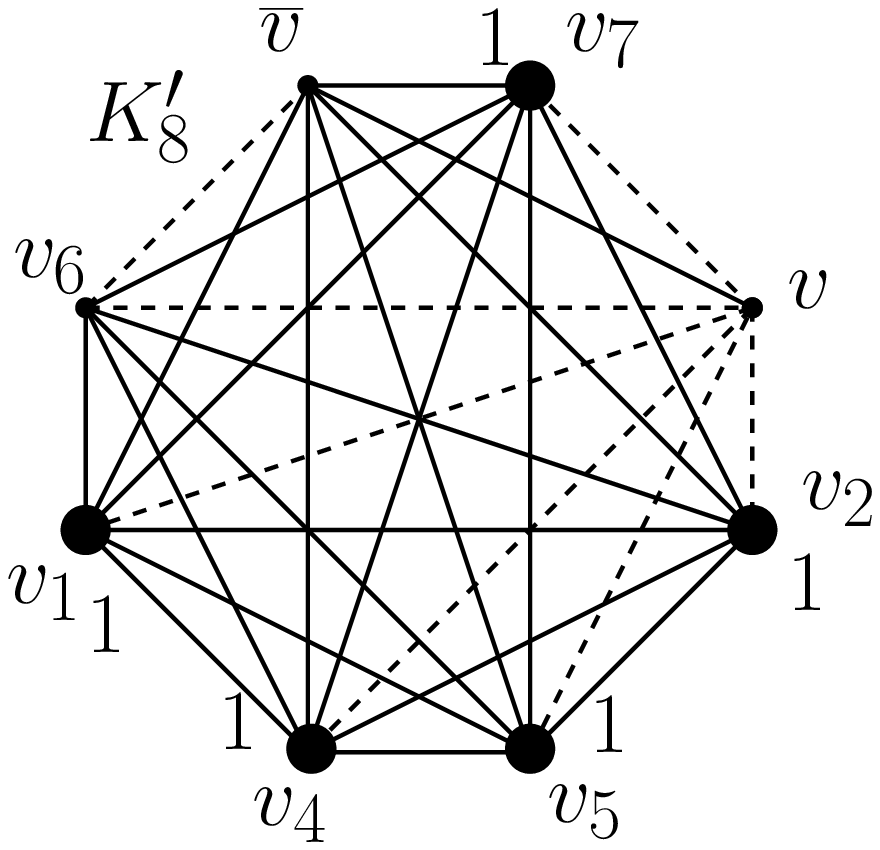}~b)
\hfill
\includegraphics[width=3.5cm,height=3.5cm]{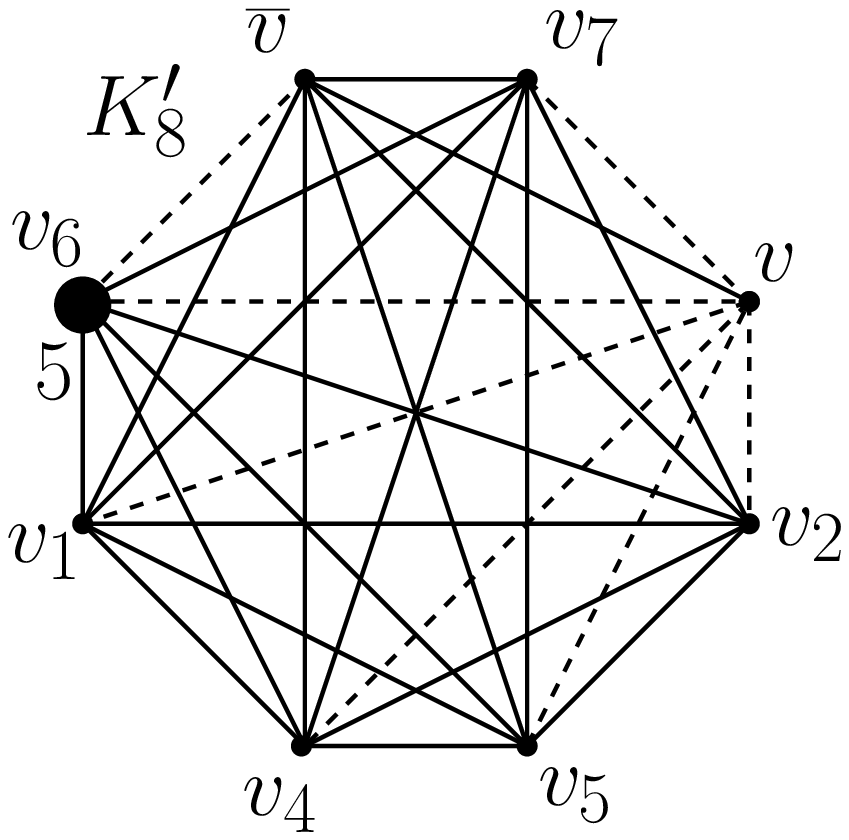}~c)
\caption{The graph $K_8'$.The omitted edges are dashed. In \textbf{a} it is depicted the divisor $D = 5(\overline{v})$. In \textbf{b} and \textbf{c} are respectively depicted the divisors $(v_1) + (v_2) + (v_4)+ (v_5) + (v_7)$ and $5(v)$ that are linearly equivalent to $D$. }
\label{FigSharp2}
\end{figure}
The divisor $D = 5(\overline{v})$ is such that 
\[
D \sim (v_1) + (v_2) + (v_4)+ (v_5) + (v_7) \sim 5(v).
\]

It is important to remark that we can conclude anyway that this graph has gonality $d-3$. In fact by construction the graph is 
\[
K_{d}' = (K_{d-1} \setminus \{\textrm{one edge}\}) \cup \{\textrm{one leaf}\}.
\]
Therefore it has the gonality of $K_{d-1} \setminus \{\textrm{one edge}\}$ by Theorem \ref{completegen}, that is $d-1 - 2 = d-3$. This can be explained by noticing that the same graph can be obtained also from the previous case by removing specific configurations of edges. 

\end{proof}

The graphs constructed in the proof are obtained by taking two vertices~$v_1$ and $v_2$ and by disconnecting each of the other vertices with one of them, except for one vertex disconnected with both of them. The edges removed form a graph that consists of two trees, $T_1$ and $T_2$, where $T_i$ is a rooted tree with root $v_i$ and $k_i$ leaves, such that $k_1 + k_2 = d-1$. Only one of the leaf vertices is a vertex of both the trees and therefore not connected with $v_1$ and $v_2$.  Two possible configurations are depicted in Figure \ref{Threes}.
The valencies of $v_1$ and $v_2$ are 
\[
\textrm{val}(v_1) = d-1-k_1 \ \ \textrm{and} \ \ \textrm{val}(v_2) = d-1-k_2.
\]
\begin{rem}
The vertices $v_3, v_4, \dots, v_d$ form a $K_{d-2}$.
\end{rem}

\begin{figure}[h]
\includegraphics[width=4cm,height=3.5cm]{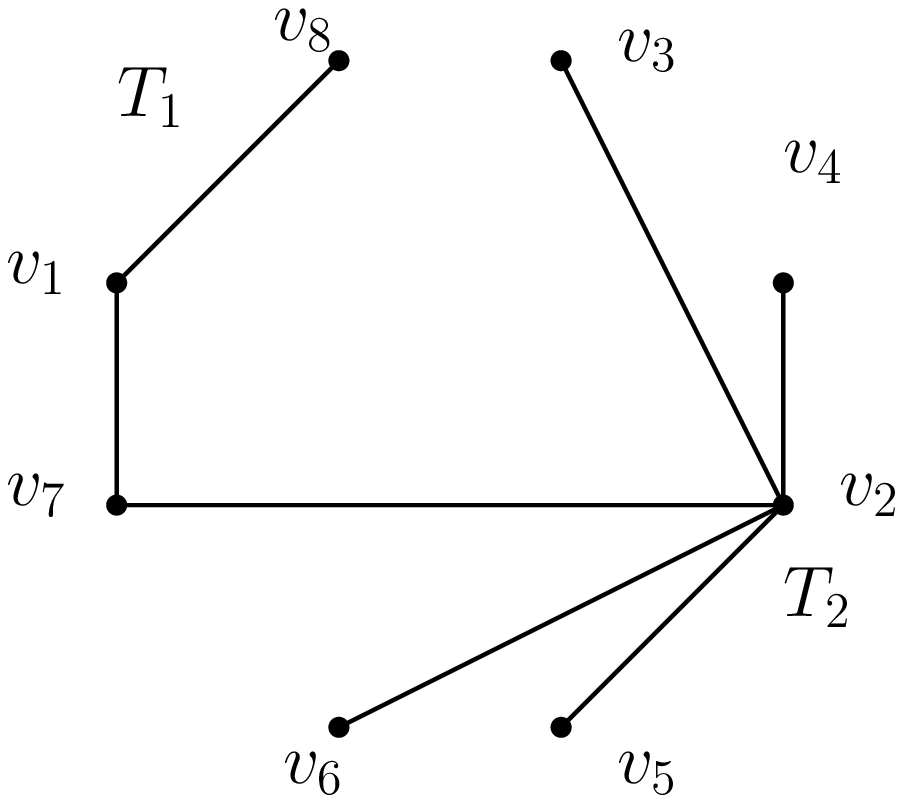}~a)
\qquad
\includegraphics[width=4cm,height=4cm]{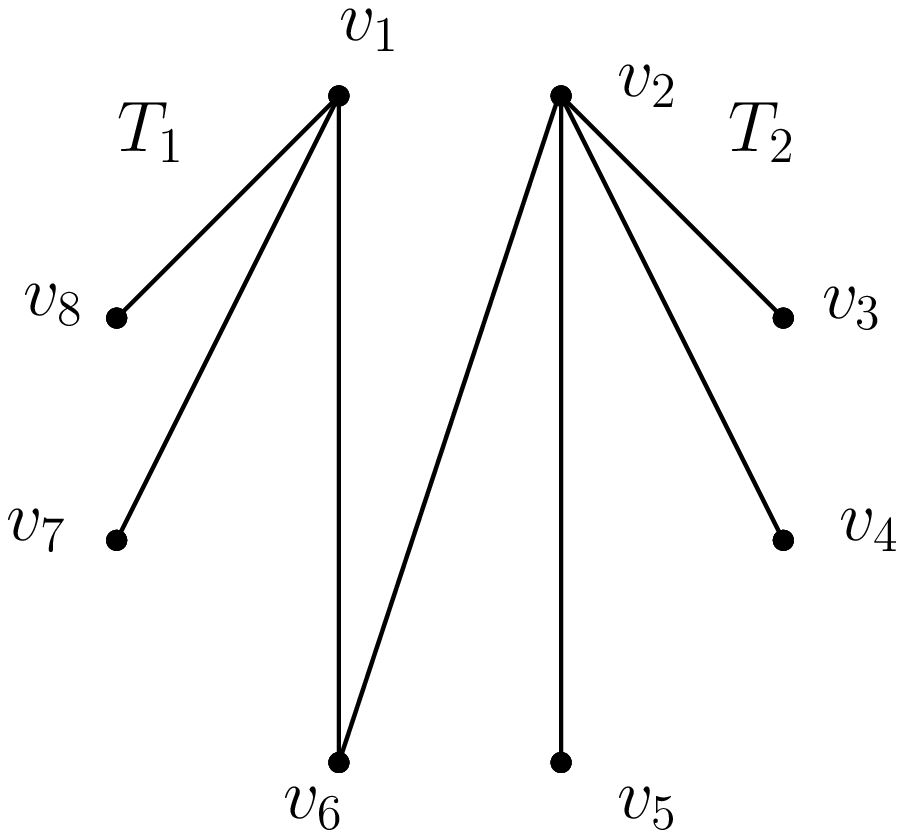}~b)
\caption{Two possible configurations of removed edges in a $K_8$ giving a graph of gonality $5$.}
\label{Threes}
\end{figure}
\end{subsection}
\end{section}

\begin{section}{Specialization map, triangulated punctured curves and skeleta} \label{sp}
In this section we will present the main tools we will use in order to give an interpretation of the combinatorial results in terms of algebraic curves. As we mentioned in the introduction, linear systems can specialize from curves to graphs as explained in \cite{Bak}. 

\begin{subsection}{Specialization map and inequality} 
Let $R$ be a complete discrete valuation ring with field of fractions $K$ and algebraically closed residue field~$k$. Let $X$ be a smooth, proper, geometrically connected curve over $K$. A {\sl regular model} for $X$ is a regular, proper, flat surface $\mathfrak{X}$ over $R$ such that the generic fiber $\mathfrak{X}_{\eta}$ is $X$. The model is {\sl semistable} if the special fiber $\mathfrak{X}_s$ is reduced, the singular points are ordinary double points and the irreducible components isomorphic to $\mathbb{P}^1$ meet the other components in at least $2$ points. If moreover the irreducible components of $\mathfrak{X}_s$ are smooth the model is {\sl strongly semistable}. The dual graph $\Gamma_{\mathfrak{X}_s}$ of the special fiber $\mathfrak{X}_s$ is constructed by considering a vertex $v_i$ for every irreducible component $c_i$ of the fiber, and an edge $e_{ij}$ connecting two vertices $v_i$ and $v_j$ for every node between the components $c_i$ and $c_j$. The semistability of the model ensures that it is possible to construct the dual graph and the strongly semistability that the dual graph has no loop edges. 

The Zariski closure in $\mathfrak{X}$ of a rational point $x \in X(K)$ intersects exactly one irreducible component of the special fiber in a smooth point. In this way to every rational point we can associate a vertex in the dual graph. Extending by linearity we obtain the {\sl specialization map} 
\[
\rho: \textrm{Div}(X(K)) \rightarrow \textrm{Div}(\Gamma_{\mathfrak{X}_s}).
\]
This map is compatible with field extension, as we explain now. Let $K'$ be a field extension of $K$, $R'$ its valuation ring, and $X_{K'} = X \times_K K'$. There exists a unique relatively minimal regular semistable model $\mathfrak{X}'$ over $R'$ which dominates $\mathfrak{X} \times_{R} R'$. The dual graph of the special fiber of $\mathfrak{X}'$ is obtained by subdividing every edge of $\Gamma_{\mathfrak{X}_s}$ in $e$ parts, where $e$ is the ramification index of $K'/K$. So, if we assign length one to all the edges of the dual graph $\Gamma_{\mathfrak{X}_s}$, giving it the structure of a metric graph, a field extension will correspond to adding $e-1$ vertices in every edges all at distance $\frac{1}{e}$. In this way, let $\overline{K}$ be an algebraic closure of $K$, we obtain a well-defined map 
\[
\tau: X(\overline{K}) \rightarrow \Gamma,
\]
which we can extend by linearity to a homomorphism 
\[
\tau_{*}: \textrm{Div}(X_{\overline{K}}) \cong \textrm{Div}(X(\overline{K})) \rightarrow \textrm{Div}(\Gamma).
\]
By specialization, the rank of the divisor can only increase: 
\begin{thm}[Specialization Lemma, Corollary 2.10 in \cite{Bak}] Let $D$ be a divisor on $X_{\overline{K}}$. The following inequality holds:
\[
\textrm{rk}_{\Gamma}(\tau_{*}(D)) \geq  \textrm{rk}_{X_{\overline{K}}}(D).
\]
\end{thm}
The framework introduced by Baker works better in the case of maximally degenerate curves, that is when all the irreducible components of the special fiber have genus zero. This is the case we deal with in this paper. In order to make this theory also interesting in the case when the components have positive genus, extra structure has been added to the graph in order to carry and take care of the data of the genera through the specialization, leading to the definitions of augmented metric graphs and metrized complexes of curves. We refer the interested reader to \cite{AC} and \cite{AB} for more details.

\end{subsection}

\begin{subsection}{Lifting problems}
As we mentioned in the introduction a natural question that arise in this context is whether for divisor on a metric graph there exists a curve and a divisor on it of the same rank specializing to it. More precisely, 
\begin{ques} \label{question1}
 Let $D$ be a divisor on a metric graph $\Gamma$. Does there exist a curve $X$, a regular model for it with dual graph $\Gamma$, and a divisor $\widetilde{D} \in \textrm{Div}(X_{\overline{K}})$, such that $\tau_* (\widetilde{D}) = D$ and $\textrm{rk}_X(\widetilde{D}) = \textrm{rk}_{\Gamma}(D)$? 
\end{ques}

\begin{defi} \label{lifting} Let $D$ be a divisor on a metric graph $\Gamma$. We say that $D$ is {\sl liftable} if the previous question has a positive answer. In the notation above, the divisor $\widetilde{D}$ is called a {\sl lift} of $D$. 
\end{defi}
We will address the problem whether the divisors constructed in some of the combinatorial results of Section \ref{2} are liftable. Other lifting problems addressed in the literature question themselves if given a metric graph there exists a curve such that any divisor on the graph is liftable, see for example \cite{KY1, KY2}.
\begin{ques} \label{question2}
 Let $\Gamma$ be a metric graph. Does there exist a curve $X$, a regular model for it with dual graph $\Gamma$, such that for every divisor $D \in \textrm{Div}(\Gamma)$ there exists a divisor $\widetilde{D} \in \textrm{Div}(X_{\overline{K}})$, such that $\tau_* (\widetilde{D}) = D$ and $\textrm{rk}_X(\widetilde{D}) = \textrm{rk}_{\Gamma}(D)$? 
\end{ques}
\begin{rem} We remark that the lift $\widetilde{D}$ is a divisor on the generic fiber which is smooth and it specializes to a divisor on the special fiber which is singular. Anyway we do not have any information about the rank of the divisor on the special fiber. 
\end{rem}

\end{subsection}

\begin{subsection}{Triangulated punctured curves and their skeleta}\label{skeleta} It is possible to associate to a curve a metric graph through a {\sl semistable vertex set} using Berkovich analytification. We will briefly introduce this construction and we refer to \cite[\S5]{BPR} and \cite[\S3]{ABBR1} for further reading. 

Let $K$ be an algebraically closed field, complete with respect to a nontrivial non-Archimedean valuation. Let $R$ be its valuation ring, $k$ its residue field and $\Lambda$ its value group. Let $X$ be a smooth, proper, connected, algebraic curve over $K$, and $D\subset X(K)$ a finite set of closed points, which we call {\sl punctures}. With the symbol $X^{\textrm{an}}$ we indicate the analytification of the curve $X$ as defined in \cite{Ber} (see also \cite[\S3]{BPR}). 

\begin{defi} A {\sl semistable vertex set of $X$} is a finite set $V$ of type 2 points of $X^{\textrm{an}}$ (see \cite{BPR}), such that $X^{\textrm{an}} \setminus V$ is a disjoint union of open balls and finitely many generalized open annuli. A {\sl semistable vertex set of $(X,D)$}~is a semistable vertex set of $X$ such that the points of $D$ are contained in distinct open balls of $X^{\textrm{an}} \setminus V$.

A {\sl triangulated punctured curve} $(X, V \cup D)$ is a smooth, connected, proper algebraic curve $X$ over $K$ with a finite set $D$ of punctures and a semistable vertex set $V$ of $(X,D)$. 
\end{defi}
By definition, a semistable vertex set induces a decomposition 
\[
X^{\textrm{an}} \setminus (V \cup D) = A_1 \cup A_2 \cup \dots \cup A_n \cup \bigcup_{j} B_j,
\]
where the $A_i$ are generalized open annuli and $B_i$ open balls. 
\begin{defi} The skeleton $\Sigma(X, V \cup D)$ is the subset 
\[
\Sigma(X, V \cup D) = V \cup D \cup \bigcup_{i=1}^n \Sigma(A_i),  
\]
where, using the notation of \cite[\S3]{ABBR1}, the generalized open annulus $A_i$ is a $K$-analytic space isomorphic to $S(a_i)_{+} = \textrm{val}^{-1}((0, \textrm{val}(a_i))) \subseteq \mathbb{A}^{1, \textrm{an}}$, for a certain $a_i$ in the maximal ideal of $R$; and its skeleton is $\Sigma(A_i)$ is identified with the interval $(0, \textrm{val}(a_i))$. 
The notation $\textrm{val}$ indicates the valuation map on the analytification of the affine line 
\[
\textrm{val}: \mathbb{A}^{1,\textrm{an}} \rightarrow \mathbb{R} \cup \{\infty\}, \, \textrm{val}(x) = -\log |x|.
\] 
\end{defi}
The skeleton $\Sigma(X, V \cup D)$ of a curve $X$ has therefore the structure of a $\Lambda$-metric graph with vertex set $V \cup D$ (where the vertices in $D$ are infinite), and edges corresponding to the skeleton of the different annuli in the decomposition. The length of the edges is given by the modulus of the annulus. 

There is a canonical embedding $\tau: \Sigma(X, V \cup D) \rightarrow X^{\textrm{an}}$ and a retraction map $\textrm{red}: X^{\textrm{an}} \rightarrow \Sigma(X, V \cup D)$, that sends a point $p \in X^{\textrm{an}}$ not in the skeleton to its connected component in $X^{\textrm{an}} \setminus \Sigma(X, V \cup D)$. 

\vspace{\baselineskip}
The reverse holds too and metric graphs can be lifted to triangulated punctured curves: 
\begin{thm}[Theorem 3.24 of \cite{ABBR1}] Let $\Gamma$ be a $\Lambda$-metric graph. There exists a triangulated punctured curve $(X,V \cup D)$ such that the skeleton $\Sigma(X, V \cup D)$ is isomorphic to $\Gamma$. 
\end{thm}
The result in \cite{ABBR1} is stated using the terminology of metrized complexes. In fact, the structure of the skeleton of a curve can be enriched to a metrized complex. For the purposes of this paper, anyway, it is enough to only deal with metric graphs.  

\vspace{\baselineskip}
We recall that we are assuming $\overline{K} = K$. The specialization map $\tau: X(K) \rightarrow \Gamma_{\mathfrak{X}_s}$ can also be obtained via the analytification $X^{\textrm{an}}$  of the curve $X$, without requiring that the model $\mathfrak{X}$ is regular. In fact if $\mathfrak{X}$ is a semistable model of $X$, by defining the length of an edge $e$ of the dual graph $\Gamma_{\mathfrak{X}_s}$, corresponding to a singular point $x$ of the special fiber, as the modulus of the open annulus $\textrm{red}^{-1}(x)$, we obtain a skeleton for $X^{\textrm{an}}$. The graph $\Gamma_{\mathfrak{X}}$ is canonically embedded in $X^{\textrm{an}}$. As we said before, there is a retraction map $\tau: X^{\textrm{an}} \rightarrow \Gamma_{\mathfrak{X}_s}$ that induces by linearity the map $\tau_{*}: \textrm{Div}(X) \rightarrow \textrm{Div}(\Gamma_{\mathfrak{X}_s})$. 

It is important to remark that the skeleta defined in terms of semistable models and the ones defined in terms of semistable vertex sets, as we introduced in \ref{skeleta} are in bijection. This follows from a bijection between the set of semistable models and the set of semistable vertex sets of a curve $X$ (see Theorem 5.38 of \cite{BPR}).  
\end{subsection}

\end{section}

\begin{section}{Plane curves with singularities} \label{pc}
Let $K$ be an algebraically closed field complete with respect to a non-Archimedean valuation and $R$ its valutation ring. The specialization inequality tells us that if $X$ is a smooth, proper, connected curve and $\mathfrak{X}$ a strongly semistable model of $X$, with special fiber $\mathfrak{X}_s$ and dual graph $\Gamma_{\mathfrak{X}_s}$, the gonality of $X$ is an upper bound for the gonality of $\Gamma_{\mathfrak{X}_s}$. So now we want to construct a smooth curve $X$ and a strongly semistable model for it, such that the dual graph of the special fiber is $K_d'$ and the gonality of $X$ is the gonality of the $K_d'$. 

\begin{subsection}{Smooth plane curves} \label{pcs} The complete graph $K_d$ is the dual graph of the special fiber of a model associated to a smooth plane curve of degree~$d$. Indeed we consider the irreducible curve $X$ over $K$ given by the following equation in projective coordinates $x$, $y$ and $z$:
\begin{equation} \label{planecurve}
F(x,y,z) = t \cdot f(x,y,z) + \prod_{i=1}^d l_i(x,y,z), 
\end{equation}
where $f(x,y,z)$ is a general polynomial of degree $d$ with coefficients with non-negative valuation. The polynomials $l_i(x,y,z)$ are of degree one with coefficient with non-negative valuation, such that modulo the maximal ideal they give the equations in $k$ of $d$ lines in the plane, no three of which are concurrent. Finally $t$ is an element of the maximal ideal of $R$. Since the polynomial $f$ is general, the curve $X$ is smooth. We consider the surface $\mathfrak{X}$ over $R$ given by the same equation (\ref{planecurve}).

It is easy to check that the surface $\mathfrak{X}$ is a model for the smooth plane curve $X$ and that the special fiber $\mathfrak{X}_s$ is the union of $d$ lines each of which intersecting the other in double points, so the dual graph is the complete graph $K_d$ on $d$ vertices.
 
The smooth plane curve $X$ of degree $d$ has gonality $d-1$ (see \cite{Nam} for the same statement for smooth complex plane curve). By the specialization inequality we know that the gonality $d-1$ of the dual graph is a lower bound; moreover every pencil of lines passing through a point on $X$ cuts out a $g^1_{d-1}$. We recall from Example \ref{exgon} that for every subset $J$ of $d-1$~vertices, the divisor $D_J = \sum_{i\in J} (v_i)$ on $K_d$ has rank $1$ and degree~$d-1$. We can construct a divisor on the curve $X$ that specialize to a divisor $D_J$. We consider a point on the curve specializing to the component corresponding to the vertex not in $J$. It is possible to find a line through this point intersecting the curve in $d-1$ points specializing to the other components $l_i$.  The $d-1$ points specialize to the divisor $D_J$. This gives us the following result (see Definition \ref{lifting}):
\begin{prop} There exists a smooth curve $X$ and a model for it, such that the dual graph of the special fiber is $K_d$ and the gonality of $X$ is the gonality of $K_d$. The divisors $D_J$ is liftable to divisors on the curve.
\end{prop} 
\end{subsection}

\begin{subsection}{Singular points of multiplicity $h$} \label{mulh} Theorem \ref{hconf} states that the gonality of $K_d'=K_d \setminus K_h$ is $d-h$. Let $X$ be the irreducible curve over $K$ given by the following equation: 
\begin{equation} \label{curvasing}
 G(x,y,z) = t \cdot g(x,y,z) + \prod_{i=1}^d l_i(x,y,z),
\end{equation}
where $g(x,y,z)$ is a general polynomial of degree $d$ with non-negative valued coefficients having a singular point of multiplicity $h$. The polynomials $l_i$ are of degree one with coefficients with non-negative valuation such that modulo the maximal ideal they give the equations with coefficients in $k$ of $d$ lines in the plane, $h$ of which intersect in the singular point of $g$ and intersect the others $d-h$ lines in ordinary double points. 

Let $\mathfrak{X}$ be the model over $R$ given by the equation (\ref{curvasing}).
By taking the closure of the singular point in $\mathbb{P}_{\mathbb{Z}}^2 \times \textrm{Spec}(R)$ and blowing up along it, we obtain a strongly semistable model $\mathfrak{X}'$ of a smooth curve $X'$, with special fiber $\mathfrak{X'}_s$ such that the $h$ lines that were intersecting in the $h$-ple point, now are not intersecting anymore. The dual graph of the special fiber is exactly~$K_d \setminus K_h$.

The curve $X'$ has gonality $d-h$, because the gonality $d-h$ of the dual graph $K_d \setminus K_h$ is a lower bound. Moreover the lines through the singular point of multiplicity $h$ induce a $g^1_{d-h}$. See \cite{OS} for the same result for irreducible plane curve over the complex number field. If we indicate with $l_1, l_2, \dots, l_h$ the lines intersecting in the point of multiplicity $h$, then the $K_h$ removed has as vertex set $\{v_1,  v_2, \dots, v_h\}$. The divisor of rank $1$ described in the proof of Proposition \ref{complete} is $D = (v_{h+1}) + (v_{h+2}) + \cdots + (v_d)$. Again there is a line through the point of multiplicity $h$ intersecting the curve in points specializing to the other components, inducing a divisor on $X'$ that specializes to $D$. Therefore we have the following proposition:
\begin{prop} There exists a smooth curve $X'$ and a model for it, such that the dual graph of the special fiber is $K_d\setminus K_h$ and the gonality of $X'$ is the gonality of the graph. The divisor $D = (v_{h+1}) + (v_{h+2}) + \cdots + (v_d)$ is liftable to a divisor on the curve. 
 \end{prop}
\end{subsection}

\begin{subsection}{Plane curves with nodes} \label{nodes} The next step is considering the case when we remove $i$ edges not forming any $K_h$ configuration, with $1 \leq i \leq d-2$ and $h\geq3$. The combinatorial result about the gonality of such a graph can also be interpreted in terms of plane curves with singularities. Let $Y$ be the reduced curve in the plane represented by the $d$ lines $l_j, \ 1 \leq j \leq d$ each of which intersects the others in ordinary double points. We fix $i$ of these nodes and by Proposition 2.11 and Theorem 2.13 of \cite{Tan}, there exist a deformation of $Y$ over a smooth irreducible curve $T$ of finite type over an algebraically closed field $K$ of characteristic zero, such that the other fibers are plane curves with only $i$ nodes specializing to the nodes we fixed in the fiber $Y$. Moreover we can take the generic curve $X$ on this family to be irreducible. Indicating $t_0 \in T$ the point such that the fiber at $t_0$ is $Y$, we have a model over $R = \mathcal{O}_{T, t_o}$ of $X$. By blowing up the assigned nodes in the family we obtain a smooth generic fiber $X'$ and a special fiber such that the dual graph is $K_d$ with $d-2$ edges removed. Therefore we obtained a smooth curve $X'$ and a strongly semistable model~$\mathfrak{X}$ such that the dual graph of the special fiber is $K_d'$.  

\vspace{\baselineskip}
By the results of Coppens and Kato in \cite{CK} the gonality of the normalization of integral plane curves of degree $d\geq 3$, with up to $d-2$ nodes  is $d-2$. Moreover if $d \geq 2k+3$ and $i \leq kd- (k+1)^2 +2$  for some $k >0$, then the $g^1_{d-2}$ is cut out by a pencil of lines: Given $p: X' \rightarrow X$ the normalization and $P$ a pencil of line, they denote 
\[P.X' = \{ p^{-1}(D.X) \ : D \in P\} \ \ \textrm{and} \ \ F(P.X') = \cap \{E \ : \ E \in P.X'\}.
\]
and they define $\{E - F(P.X') \ : \ E \in P.X'\}$ the linear system induced on $X'$ by $P$. The divisors of rank~$1$ constructed in the proof of Theorem \ref{completegen} are of the form $D_{jk} = \sum_{i \not = j,k} (v_i)$, where $e_{jk} = v_jv_k$ is a removed edge. Consider the node of $X$ specializing to the node of intersection of the lines $l_j$ and $l_k$ on the fiber $Y$. By taking a pencil of lines through it, we can find a line that induces a divisor on $X'$ that specializes to $D_{jk}$. We proved the following result:
\begin{prop} There exists a smooth curve $X'$ and a model for it, such that the dual graph of the special fiber is $K_d'$ and the gonality of $X'$ is the gonality of $K'_d$. The divisors $D_{jk}$ are liftable to divisors on the curve.
\end{prop}

In the just mentioned paper of Coppens and Kato it is also proved that the normalization of an integral plane curve of degree $d\geq 6$ with $d-1$ nodes has gonality $d-2$. So we cannot lift the graph $K_d'$ with $d-1$ edges removed, constructed in the proof of Theorem \ref{stable} using plane nodal models,  since they have gonality $d-3$. In the following section we will give a geometric interpretation using different techniques. 
\end{subsection}
\end{section}

\begin{section}{Lifting Harmonic Morphisms of Metric Graphs} \label{hm}
In order to give a geometric interpretation of Theorem \ref{stable} we use results from \cite{ABBR1} and \cite{ABBR2} about lifting of harmonic morphism of metric graphs to curves.  
\begin{subsection}{Harmonic morphisms} 
Harmonic morphism between finite graphs were studied by Baker and Norine in \cite{BN2}. Here we recall some definitions regarding the theory of harmonic morphisms between $\Lambda$-metric graphs following \cite{ABBR1} and \cite{ABBR2}, to which we refer for further reading.
\begin{defi} Let $\Gamma'$ and $\Gamma$ be $\Lambda$-metric graphs, $V(\Gamma')$ a vertex set of $\Gamma'$ and $V(\Gamma)$ a vertex set of $\Gamma$. A continuous map $\varphi: \Gamma'  \rightarrow \Gamma$ is a {\sl $(V(\Gamma'), V(\Gamma))$-morphism of $\Lambda$-metric graphs} if $\varphi(V(\Gamma')) \subseteq V(\Gamma)$, $\varphi^{-1}(E(\Gamma)) \subset E(\Gamma')$, and $\varphi$ restricted to an edge $e' \in E(\Gamma')$ is a dilation by some factor $d_{e'}(\varphi) \in \mathbb{Z}_{\geq 0}$. The continuous map $\varphi: \Gamma' \rightarrow \Gamma$ is called a {\sl morphism} if there exist vertex sets $V(\Gamma')$ and $V(\Gamma)$ of respectively $\Gamma'$ and $\Gamma$ such that $\varphi$ is a $(V(\Gamma'), V(\Gamma'))$-morphism. We say that $\varphi$ is {\sl finite} if $d_{e'}(\varphi) >0$ for every edge $e'$ of $\Gamma'$. The integer $d_{e'}(\phi)$ is called the {\sl degree of $\varphi$ along $e'$}.
\end{defi}

\begin{rem} For an edge $e'$ of $\Gamma'$, we have $\varphi(e') = v$, that is $e'$ is contracted to the vertex $v$, if and only if $d_{e'}(\varphi) = 0$. So the map is finite if no edges are contracted; equivalently if the preimage of every point of $\Gamma$ consists of a finite number of points.  
\end{rem}

\begin{defi} Let $\varphi: \Gamma' \rightarrow \Gamma$ be a morphism. Given $p' \in \Gamma'$, $v' \in T_{p'}(\Gamma')$ a tangent direction and $e' \in E(\Gamma')$ the edge in the direction of $v'$, the {\sl directional derivative $d_{v'}(\varphi)$ of $\varphi$ in the direction $v'$} is defined as $d_{v'}(\varphi) := d_{e'}(\varphi)$. 
Let $p=\varphi(p')$. The morphism $\varphi$ is {\sl harmonic at $p'$} if for every tangent direction $v \in T_p(\Gamma)$ the integer number
\[
d_{p'}(\varphi) = \sum_{\substack{v' \in T_{p'}(\Gamma), \\  \varphi(v') = v}} d_{v'}(\varphi)
\]
is independent of $v$. The number $d_{p'}(\varphi)$ is called the {\sl degree of $\varphi$ at $p'$}. 

If the morphism $\varphi$ is surjective and harmonic at every $p' \in \Gamma'$, we say that it is {\sl harmonic}. In this situation the number
\[
\textrm{deg}(\varphi) = \sum_{p' \in \Gamma', \varphi(p') = p} d_{p'}(\varphi) 
\]
does not depend on $p$ and it is called the {\sl degree of $\varphi$}. 
\end{defi}

The collection of metric graphs, together with harmonic morphisms between them, forms a category.

\begin{defi} Let $\varphi: \Gamma' \rightarrow \Gamma$ be a harmonic morphism of metric graphs. The {\sl ramification divisor} of $\varphi$ is the divisor $R = \sum_{p \in \Gamma'} R(p')(p')$ whose coefficient at the point $p' \in \Gamma$ is given by 
\[
R(p') = 2 \Big( d_{p'}(\varphi) -1 \Big) - \sum_{v' \in T_{p'}(\Gamma')} \Big( d_{v'}(\varphi) -1 \Big).
\]

Given a vertex $p' \in V(\Gamma')$ with $d_{p'}(\varphi) \not =0$ we define the {\sl ramification degree} of $\varphi$ at $p'$ to be 
\[
r_{p'} = R(p') - \Big|\Big\{ v' \in T_{p'}(\varphi) \ | \ d_{v'}(\varphi) =0 \Big\}\Big|.
\]
We say that a harmonic morphism of metric graphs is {\sl effective} if $r_{p'} \geq 0$ for every $p' \in \Gamma'$. 
\end{defi}

\medskip 

Let $K_d' = K_d \setminus \{ e_1, e_2, \dots, e_{d-1} \}$ be a graph of the form obtained in the proof of Theorem \ref{stable}. The edges removed form two trees, $T_1$ and $T_2$, each of which is a rooted tree with $k_i$ leaves, such that $k_1 + k_2 = d-1$. 
The roots are $v_1$ and $v_2$. Only one of the leaf vertices is a vertex of both the trees and therefore not connected with $v_1$ and $v_2$.  We label it with~$v_{d-k_1 +1} = v_{k_2+2}$ . We indicate the vertices adjacent to $v_1$ different form $v_2$ with $v_3, v_4, \dots, v_{d-k_1}$ and the vertices adjacent to $v_2$ different from~$v_1$ with $v_{k_2+3}, \dots, v_d$.  

From now on we will visualize the graph $K_d'$ as showed in Figure \ref{NewGraph}, since it will be more effective for our arguments. Note that we are not drawing the removed edges anymore. Moreover we will suppose $k_1, k_2 \geq 1$, otherwise we can contract the edge $v_1v_2$ and reduce to the case $K_{d-1}'$. 

We consider the graphs $K_d'$ such that the edges $v_{k_2+2}v_i$ with $i \geq 3$ have all the same length and the same holds for the edges $v_1v_i$ with $3 \leq i \leq d-k_1$ and the edges $v_2v_i$ with $k_2+3 \leq i \leq d$. 

\begin{figure}[h]
\includegraphics[width=6cm,height=3cm]{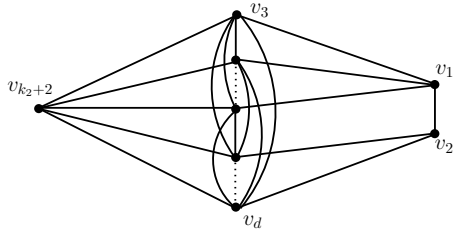}
\caption{The graph $K_d'$. The omitted edges are not depichted anymore.}
\label{NewGraph}
\end{figure}

\begin{prop} \label{harmonic}
There exists a harmonic morphism of metric graphs of degree $d-3$ from $K_d'$ to a metric tree $T$.
\end{prop}
\begin{proof}
Let $T$ be the metric tree consisting of three vertices $u_1$, $u_2$ and $u_3$ such that $u_1$ is adjacent to $u_2$ and $u_2$ is adjacent to $u_3$. We define the morphism $\varphi: K_d' \rightarrow T$ in the following way
\[
\varphi(v) = \begin{cases}  u_1 \ \ \textrm{if} \ \ v = v_{k_2+2} \\ 
             u_2 \ \ \textrm{if} \ \ v \in \{v_3,v_4, \dots,v_{d-k_1},v_{k_2+3}, \dots,v_d \} \\
             u_3 \ \ \textrm{if} \ \ v \in \{v_1, v_2 \}. \\
            \end{cases}
\]
From the definition of the map on the vertices we have that the map on the edges is 
\[
\varphi(v_iv_j) = \begin{cases}  u_1u_2 \ \ & \textrm{if} \ \ i = k_2+2 \ \textrm{and} \ j \in \{3, 4, \dots, d-k_1, k_2+3,\dots, d\} \\ 
             u_2 \ \ & \textrm{if} \ \ i,j \in \{3, 4, \dots,d-k_1,k_2+3, \dots,d \} \\
             u_2u_3 \ \ & \textrm{if} \ \ i = 1 \ \textrm{and} \ j \in \{3, 4, \dots,d-k_1\} \ \textrm{or} \\
             & \textrm{if} \ \ i = 2 \ \textrm{and} \ j \in \{ k_2+3,\dots, d\} \\
             u_3 \ \ & \textrm{if} \ \ i=1 \ \textrm{and} \ j =2. \\
            \end{cases}
\]
By definition $d_{e'}(\varphi)=0$ for all the edges that are mapped to a vertex; for all the other edges instead we set $d_{e'}(\varphi)=1$. 
The morphism is illustrated in Figure \ref{Harmonic}.
  
\begin{figure}[h]
\includegraphics[width=7cm,height=4cm]{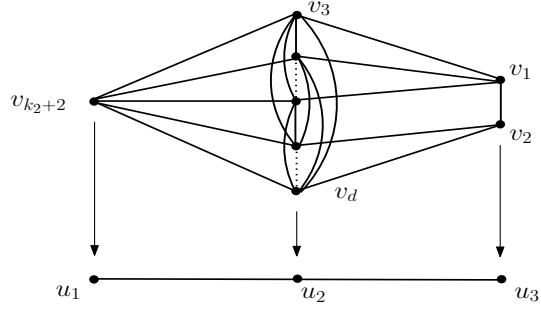}
\caption{The harmonic morphism $\varphi: K_d' \rightarrow T$.}
\label{Harmonic}
\end{figure}

The morphism $\varphi: K_d' \rightarrow T$ defined above is harmonic since it is clearly surjective and moreover at every $p' \in K_d'$ we have 
\[
d_{p'}(\varphi) =  \begin{cases}  d-3 \ \ & \textrm{if} \ \ p' = v_{k_2+2} \\ 
             1 \ \ & \textrm{if} \ \ p' \in \{v_3,v_4, \dots,v_{d-k_1},v_{k_2+3}, \dots,v_d \} \\
             d - k_1 - 2 \ \ & \textrm{if} \ \ p' = v_1 \\
             d - k_2 - 2 \ \ & \textrm{if} \ \ p' = v_2 \\
             1 \ \ & \textrm{if} \ \ p' \in v_{k_2 +2}v_j \ \ \textrm{with} \ j \in \{3,4, \dots,d-k_1,k_2+3, \dots,d \} \\
             1 \ \ & \textrm{if} \ \ p' \in v_jv_1 \ \ \textrm{with} \ j \in \{3,4, \dots,d-k_1 \} \ \ \textrm{or} \\ 
             1 \ \ & \textrm{if} \ \ p' \in  v_jv_2 \ \ \textrm{with} \ j \in \{k_2+3, \dots,d \} \\
             0 \ \ & \textrm{if} \ \ p' \in \ \textrm{int}(e) \ \textrm{with $e$ a contracted edge}, \\
              \end{cases}
 \]
and it is independent of the tangent direction $v \in T_{\varphi(p')}(T)$. 

The degree of the morphism $\varphi$ is $d-3$, because for every $p \in T$ 
\[
 \sum_{\substack{p' \in K_d', \\ \varphi(p') =p}} d_{p'}(\varphi) =d-3.
\]
\end{proof}
\end{subsection}

\begin{subsection}{Tropical modifications and tropical curves} The morphism constructed in the proof above is clearly not finite, in this section we will modify the graphs and extend the morphism in order to obtain a finite one. This is necessary because, as we will see later, only for finite morphisms it makes sense to talk about lifting them to morphisms of curves. 
\begin{defi}
Let $\Gamma$ be a metric graph. An {\sl elementary tropical modification} of $\Gamma$ is a metric graph $\Gamma_1 = \Gamma \cup [0, \infty]$ obtained from $\Gamma$ by attaching the segment $[0, \infty]$ in such a way that $0 \in [0, \infty]$ is identified with a finite point~$p \in \Gamma$. 
A {\sl tropical modification} of $\Gamma$ is a metric graph obtained by a finite sequence of elementary tropical modifications.

Tropical modifications and their inverses generate an equivalence relation on the set of metric graphs. A class of equivalence is called a {\sl tropical curve}. 
\end{defi}

We indicate with $\mathbb{TP}^1$ the unique rational tropical curve.

\begin{defi}
Let $\Gamma'$ be a metric graph and $\varphi: \Gamma' \rightarrow \Gamma$ a harmonic morphism. An {\sl elementary tropical modification of $\varphi$} is a harmonic morphism $\widetilde{\varphi}: \Gamma_1' \rightarrow \Gamma_1$, where $\phi': \Gamma_1' \rightarrow \Gamma'$ and $\phi: \Gamma_1 \rightarrow \Gamma$ are elementary tropical modifications and $\phi' \circ \varphi = \widetilde{\varphi} \circ \phi$. A sequence of elementary tropical modification is called a {\sl tropical modification}. 

A {\sl tropical morphism of tropical curves} $\varphi: C' \rightarrow C$ is a harmonic morphism of metric graphs between some representatives of $C'$ and $C$ considered up to tropical equivalence and that has a finite representative. 
\end{defi}

\begin{defi} A tropical curve $\Gamma$ is {\sl $d$-gonal} if there exists a tropical morphism $\varphi: \Gamma \rightarrow \mathbb{TP}^1$ of degree $d$.
\end{defi}

\begin{thm} \label{harmonic2} The tropical curve of which $K_d'$ is a representative, is $(d-3)$--gonal.
\end{thm}
\begin{proof} It is possible through tropical modifications to obtain a finite effective harmonic morphism $\widetilde{\varphi}: \widetilde{K_d'} \rightarrow \widetilde{T}$. Let $e' \in K_d'$ be a contracted edge and $u \in V(T)$ its image. We construct the graph $\widetilde{K_d'}$ with two tropical modifications in the mid point $m$ of the edge $e'$, so it has two new edges $e_1'$, $e_2'$ obtained by adding a vertex in the edge $e'$ and two edges $e_3'$ and $e_4'$ corresponding to $[0, \infty]$.  We consider the graph $\widetilde{T}$ obtained by a tropical modification in $u$, so it has a new edge $f$ corresponding to the segment $[0, \infty]$. Moreover we insert another vertex $u'$ in the edge $f$ at distance equal to the length of $e_1'$ and $e_2'$,  obtaining the edge $uu'$ and an infinite edge starting at $u'$, that we indicate again with $f$. The tropical modifications are depicted in Figure \ref{finite}.
\begin{figure}[h]
\includegraphics[width=3cm,height=3cm]{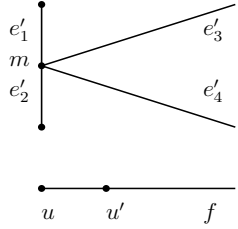}
\caption{The tropical modifications at the point $m$ and $u$ constructed in the proof of Theorem \ref{harmonic2}.}
\label{finite}
\end{figure}

We define the map $\widetilde{\varphi}$ in the following way: 
\begin{itemize}
\item $ \widetilde{\varphi}(e_1') = \widetilde{\varphi}(e_2') = uu'$; 
\item $ \widetilde{\varphi}(e_3') = \widetilde{\varphi}(e_4') = f$;
\item $ \widetilde{\varphi}|_{\widetilde{K_d'} \setminus \{e_1', e_2', e_3', e_4'\}} = \varphi|_{K_d' \setminus \{e'\}};$
\item the degree of $\widetilde{\varphi}$ is set equal to $1$ along every new edge.  
\end{itemize}

The obtained morphism might not be harmonic at the vertices of the edge $e'$, because for the tangent direction $v$ corresponding to the edge $uu'$, the number $d_{v}(\widetilde{\varphi})$  might be smaller from the number $d_w(\widetilde{\varphi})$ for another tangent direction $w$. It can be made harmonic by adding $d_w(\widetilde{\varphi}) - d_{v}(\widetilde{\varphi})$ tropical modifications at the vertex, inserting a vertex in the new infinite edges at distance $e_1'$ and mapping them to the edges $uu'$ and $f$. Again the degree of $\widetilde{\varphi}$ along any new edge is set equal to $1$. The harmonic morphism is depicted in Figure \ref{HarmonicGonal}.

By repeating this process for every contracted edge we obtain a finite harmonic morphism $\widetilde{\varphi}: \widetilde{K_d'} \rightarrow \widetilde{T}$ that coincides with $\varphi$ outside the contracted edges. The degree of the morphism is $d-3$. 

\vspace{\baselineskip}
The harmonic morphism $\widetilde{\varphi}$ constructed above is effective because 
\[
\begin{array}{ccl}
r_{p'} &=& R(p') - \Big|\Big\{v' \in T_{p'}(\widetilde{\varphi}) \ | \ d_{v'}(\widetilde{\varphi}) =0\Big\}\Big| \\
&=& 2 d_{p'}(\widetilde{\varphi})- 2 - \Big|\Big\{v' \in T_{p'}(\widetilde{\varphi}) \ | \ d_{v'}(\widetilde{\varphi}) =0 \Big\}\Big| \\
\end{array}
\]
and $r_{p'} \geq 0$ for every $p' \in \widetilde{K_d'}$, since $d_{v'}(\widetilde{\varphi})=1$ for every $v' \in \widetilde{K_d'}$. 

\vspace{\baselineskip}
\begin{figure}[h]
\includegraphics[width=7 cm,height=4cm]{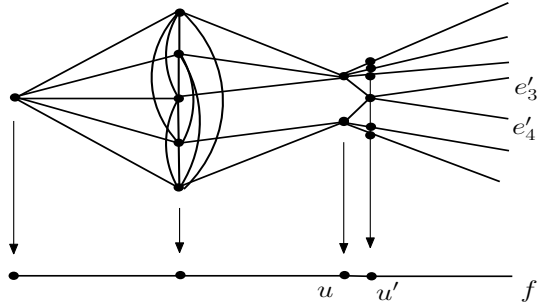}
\caption{The harmonic morphism $\widetilde{\varphi}: \widetilde{K_d'} \rightarrow \widetilde{T}$ constructed in the proof of Theorem \ref{harmonic2}.}
\label{HarmonicGonal}
\end{figure}
\end{proof}

To make the notation easier we indicate the tropical morphism of metric graphs $\widetilde{\varphi}$ constructed above with $\varphi : K_d' \rightarrow T$ and we use the same notation for the morphism of tropical curves $\varphi: K_d' \rightarrow \mathbb{TP}^1$. 
\begin{rem} The construction above is explained in the proof of Proposition 3.20 of \cite{ABBR2}. 
\end{rem}
\end{subsection}

\begin{subsection}{Lifting harmonic morphisms}
We want to lift the tropical morphism $\varphi: K_d' \rightarrow \mathbb{TP}^1$ constructed in the previous proof to a finite morphism of degree $d-3$ from a curve to the projective line $\mathbb{P}^1$ inducing the morphism $\varphi$ on the skeleta. As we will see this boils down to prove the non-emptiness of some sets that can be associated to the harmonic morphism, then by results of \cite{ABBR1} and \cite{ABBR2}, we will then know that the morphism can be lifted. We remark that the construction of the lift passes through a lift to a morphism of metrized complexes. We will only briefly mention these in our arguments and give references to the precise results. 
\vspace{\baselineskip}

Let $K$ be an algebraically closed field, complete with respect to a nontrivial non-Archimedean valuation. Let $k$ be its residue field, and $\Lambda = \textrm{val}(K^{\times})$ its value group. All the metric graphs introduced in this Section are $\Lambda$-metric graphs. Lemma 3.15 and Remark 3.16 of \cite{ABBR1} state that given a triangulated punctured curve $(X, V \cup D)$, a finite set $D'$ of punctures and a semistable vertex set $V'$, the metric graphs underlying $\Sigma(X, V \cup D)$ and $\Sigma(X, V' \cup D')$ represent the same tropical curve. 
\begin{defi}(Definition 4.25 of \cite{ABBR1}) A {\sl finite morphism} of triangulated punctured curves $\varphi: (X', V' \cup D') \rightarrow (X, V \cup D)$ is defined as a finite morphism $\varphi: X' \rightarrow X$ such that $\varphi^{-1}(D)= D'$, $\varphi^{-1}(V) = V'$, and $\varphi^{-1}(\Sigma(X, V \cup D)) = \Sigma'(X', V' \cup D')$ (as sets).
\end{defi}
Corollary 4.26 of the same article states that given a finite morphism $\varphi$~of smooth, proper, connected $K$-curves and a finite set of punctures $D \subseteq X(K)$ and $D' = \varphi^{-1}(D)$, then there exist semistable vertex sets $V$ and $V' = \varphi^{-1}(V)$ such that $\varphi$ induces a finite morphism of punctured curves $\varphi: (X', V' \cup D') \rightarrow (X, V \cup D)$.   In particular $\varphi$ induces a finite harmonic morphism on suitable choices of skeleta. 
\begin{defi}[Definition 2.22 of \cite{ABBR2}]
A tropical morphism of tropical curves $\varphi: C' \rightarrow C$ is {\sl liftable} if there exists a finite morphism of smooth, proper, connected $K$-curves $\psi: X' \rightarrow X$ inducing $\varphi$ on the skeleta as explained above. 
\end{defi}

Let $\varphi: \Gamma' \rightarrow \Gamma$ be a finite  harmonic morphism of $\Lambda$-metric graphs. As explained in \cite{ABBR2} to every $p' \in \Gamma'$ it is possible to associate a collection of val$(\varphi(p'))$ partitions of $d_{p'}(\varphi)$: for every tangent direction $v$ of $\varphi(p)$ we define the partition $\mu_i(p')$ as all the directional derivatives $d_{v'}(\varphi)$ for $v'$ such that $\varphi(v')=v$. 
We define the set
\[
\mathcal{A}_{0, 0}^{d_{p'}(\varphi)}(\mu_1, \dots, \mu_{\textrm{val}(\varphi(p'))})
\]
as the set of tame covering $\varphi_{p'}: C' \rightarrow C$ of smooth, proper, curves over $k$, with the properties: 
\begin{itemize}
\item the curves $C'$ and $C$ are irreducible of genus zero;
\item the degree of $\varphi_{p'}$ is $d_{p'}(\varphi)$;
\item the branch locus of $\varphi_{p'}$ contains at least $\textrm{val}(\varphi(p'))$ distinct points and the ramification profile at every preimage is given by one of the partitions $\mu_i$. 
\end{itemize}
We remark that when defining these sets we think of metric graphs as totally degenerated augmented graphs. This also explains why we are considering curves of genus zero. See \cite{ABBR2} for further reading.  

\vspace{\baselineskip}

If $\mathcal{A}_{0,0}^{d_{p'}(\varphi)}(\mu_1, \dots, \mu_{\textrm{val}(\varphi(p')}) \not = \emptyset$ for every $p' \in \Gamma'$, by Proposition 3.3 of \cite{ABBR2} the morphisms $\varphi$ can be lifted to a tame harmonic morphism of metrized complexes (see Definition 2.19 and 2.21 of \cite{ABBR1}). Finally, by Proposition 3.24 of \cite{ABBR1} and Proposition 7.15 of \cite{ABBR1} we can conclude that there is a finite morphisms of curves lifting the morphism of metrized complexes. 
Combining all these results together we obtain that the harmonic morphism $\varphi: \Gamma' \rightarrow \Gamma$ is liftable to a finite morphism of triangulated punctured curves. The genera of the curves and the degree of the morphism is preserved.

\begin{thm} The morphism $\varphi: K_d' \rightarrow T$ constructed in the proof of Theorem \ref{harmonic2} is liftable to a finite morphism of triangulated punctured curves.
\end{thm}
\begin{proof}
We need to show that the sets $\mathcal{A}_{0,0}^{d_{p'}(\varphi)}(\mu_1, \dots, \mu_{\textrm{val}(\varphi(p')})$ are not empty for every $p' \in K_d'$.
From the proof of the theorem it follows that for every $p' \in K_d'$ the set of $\textrm{val}(\varphi(p'))$ partitions of the number $d_{p'}(\varphi)$ is given by trivial partitions since by construction for every $p' \in K_d'$ and all tangent directions $v$ at $\varphi(p')$ the directional derivatives $d_{v'}(\varphi)$ are equal to 1 for every $v'$ satisfying $\varphi(v')=v$. Therefore the sets are not empty and by the arguments made above we can conclude.
\end{proof}

The degree is preserved by the lift. We obtained therefore a $K$-curve $(X', V' \cup D')$ and a finite morphism of degree $d-3$ to $\mathbb{P}^1$; so $X'$ is a curve of gonality $d-3$ that has~$K_d'$ as skeleton.

The divisor of rank $1$ obtained in the proof of Theorem~\ref{stable} is 
\[
 D = (v_3)+ (v_4) + \cdots +(v_{d-k_1}) + (v_{k_2+3}) +  \cdots + (v_d).
\]
It is the divisor that appears as a fiber of the finite harmonic morphism at the point $u_2 \in T$, in fact 
\[ 
 D= \sum_{v \in \varphi^{-1}(u_2)} d_v(\varphi)(v).
\]
Therefore by Proposition 4.2 of \cite{ABBR2} and the argument just made, it can be lifted to a divisor of degree $d-3$ and rank $1$ on the curve~$X'$. 
\end{subsection}
\end{section}
\begin{subsection}*{Acknowledgments} I am very grateful to Marc Coppens and Filip Cools for suggesting me this problem, for their guidance and for their valuable comments on earlier versions of this paper. I would also like to thank Erwan Brugall\'e for interesting discussions about harmonic morphisms and the referee for helpful comments, which were valuable in improving the manuscript. 
\end{subsection}

\bibliographystyle{amsalpha}
\bibliography{bibliografia}

\end{document}